\documentclass[12pt]{amsart}
\usepackage{amsmath,amsfonts, amssymb, textcomp,longtable, xypic}%pstricks}

\newtheorem{thm}{Theorem }[section]

\newtheorem{lemma}[thm]{Lemma }

\newtheorem{prop}[thm]{Proposition }

\newtheorem{corollary}[thm]{Corollary }

\theoremstyle{definition}
\newtheorem{deff}[thm]{Definition }

\newtheorem{rem}[thm]{Remark }

\newtheorem{ex}[thm]{Example }

\newtheorem{question}[thm]{Question }

\numberwithin{equation}{section}

\def\PP{{\mathbb P}}
\def\RR{{\mathbb R}}
\def\CC{{\mathbb C}}
\def\QQ{{\mathbb Q}}
\def\ZZ{{\mathbb Z}}

\def\FF{{\mathbb F}}
\def\HH{{\mathbb H}}
\def\fq{{\mathbb{F}_q}}
\def\kk{{\bar{k}}}

\def \bra#1\ket {\mathop{\vphantom{#1}\left<\smash{#1}\right>}\nolimits}

\newcommand{\rig}{{\mathrm{rig}}}
\newcommand{\real}{{\mathrm{real}}}
\def\Die{Dieudonn\'e }

 \DeclareMathOperator{\inv}{inv} 
\DeclareMathOperator{\Hom}{Hom}
\DeclareMathOperator{\End}{End} \DeclareMathOperator{\Aut}{Aut}

\DeclareMathOperator{\Gal}{Gal} \DeclareMathOperator{\Spec}{Spec}
 \DeclareMathOperator{\rk}{rk}
\DeclareMathOperator{\Tr}{Tr} 
\DeclareMathOperator{\GL}{GL}
\DeclareMathOperator{\SL}{SL}
\DeclareMathOperator{\CSU}{ESL}
\DeclareMathOperator{\ESL}{ESL}

 \DeclareMathOperator{\Br}{Br}
\DeclareMathOperator{\NS}{NS}

\def\C{\mathcal{C}}
\def\D{\mathcal{D}}

\def\p{\mathfrak{p}}

\newcommand{\XX}{{\ensuremath{\overline{X}}}}

\def\ch{\mathop{\mathrm{char}}\nolimits}
\def\plim{\mathop{{\lim\limits_{\longleftarrow}}}\nolimits}

\newcommand \eps {\varepsilon}
\renewcommand \phi {\varphi}

\begin{document}
\author{Sergey Rybakov}
\thanks{Supported by the Israel Science Foundation,  grant No.  1405/22}
%\address{Institute for information transmission problems of the Russian Academy of Sciences}
%\address{Interdisciplinary Scientific Center J.-V. Poncelet (ISCP)}
%\address{Higher School of Modern Mathematics, MIPT }
\address{Department of Mathematics, Ben-Gurion University of the Negev, Israel}

\email{rybakov.sergey@gmail.com}%
%\dedicatory{}%
\title[Generalized Kummer surfaces]
{Generalized Kummer surfaces over finite fields}
\date{}


\begin{abstract}
We prove a refinement of the Katsura theorem on finite group actions on abelian surfaces such that the quotient is birational to a $K3$ surface. 
As an application, we compute traces of Frobenius on the Neron--Severi groups of supersingular generalized Kummer surfaces over finite fields.
\end{abstract}

\maketitle

\section{Introduction}
Throughout this paper $k$ is a perfect field of characteristic $p$, and $\kk$ is an algebraic closure of $k$.
Let $A$ be an abelian surface over $k$ with an action of a finite group $G$. 
In this paper, we study only actions such that $G$ preserves the group law on $A$; in other words, there is a group homomorphism $G\to\End(A)^*$, where for a ring $R$ we denote by $R^*$ the multiplicative group of invertible elements. 
If a minimal smooth model $X(A,G)$ of the quotient $A/G$ is a $K3$ surface, we say that $X(A,G)$ is a \emph{generalized Kummer surface}. 
In~\cite{Ry12} we classified the zeta functions of Kummer surfaces $X(A,\ZZ/2\ZZ)$ over finite fields of characteristic $p>2$. 
This paper is a step towards such a classification for generalized Kummer surfaces. 

For an algebraic variety $X$ over $k$ we denote by $\Bar X$ the base change of $X$ to $\kk$.
A $K3$ surface $X$ is called \emph{Shioda supersingular} if the rank of $\NS(\Bar X)$ is $22$.
A Shioda supersingular K3 surface is supersingular, and a supersingular Kummer surface is Shioda supersingular~\cite{Ar}.
More generally, the Tate conjecture is proved for $K3$ surfaces over finitely generated fields of odd characteristic~\cite{MP}; therefore, a $K3$ surface over such a field is Shioda supersingular if and only if it is supersingular. 

For a Shioda supersingular $K3$ surface $X$, we have the natural isomorphism for any prime $\ell\neq p$: \[\NS(\Bar X)\otimes\QQ_\ell\to H^2(\Bar X,\QQ_\ell).\eqno{(*)}\]
 We see that if the base field $k$ is finite, the zeta function of a Shioda supersingular $K3$ surface $X$ is uniquely determined by the Frobenius action on the Neron-Severi group 
 $\NS(\Bar X)$. This observation is a link between the arithmetic and geometry of a supersingular $K3$ surface.  
 
Assume that $k=\FF_q$ is a finite field.
Since odd cohomology groups of $X$ are trivial, the isomorphism $(*)$ leads to the formula for the number of points on $X$ over $k$:
\[|X(k)|=1+q\Tr_X+q^2,\eqno{(**)}\] where $\Tr_X$ is the trace of the Frobenius action on the Neron--Severi group of $X$. The main question is: what are the possible values of $\Tr_X$, where $X$ runs through all supersingular $K3$ surfaces over $\FF_q$?

Our main motivation comes from the analogy with the cubic surfaces in $\PP^3$.
For a cubic surface $X$ we have the same formula $(**)$ for the number of points. In~\cite{Ser12} Serre asked which values of $\Tr_X$ can arise for smooth cubic surfaces $X$ in $\PP^3$ over $\FF_q$. The answer was recently obtained in~\cite{BFL}. A more general question about the zeta functions of cubic surfaces was treated in~\cite{RT16,T20,LT}.
A direct generalization of Serre's question to quartics in $\PP^3$ seems very difficult, but we can try to generalize the question as follows. 
Since a quartic in $\PP^3$ with reasonable singularities is a $K3$ surface, we can study traces of Frobenius on Neron--Severi groups of $K3$ surfaces.
The isomorphism $(*)$ holds for any smooth cubic surface, but not for any $K3$ surface;
 therefore, we focus on Shioda supersingular $K3$ surfaces. In this paper we compute Frobenius traces for generalized Kummer surfaces only, but it is natural to pose the question for supersingular $K3$ surfaces as well.

The paper is organized as follows.
In Section~\ref{rigid_sec} we study rigid group actions on abelian varieties.

\begin{deff}
We say that the action of a finite group $G$ on an abelian variety $A$ is \emph{rigid} if the representation of $G$ in $V_\ell(A)$ is \emph{without fixed points}, i.e., for any $g\in G$ of order $r$ the eigenvalues of the action of $g$ on $V_\ell(A)$ are primitive $r$-th roots of
unity.
\end{deff}

\begin{rem}\label{Rem1}
If an action of $G$ on $A$ is rigid, then it is faithful.
Moreover, if $g\in G$ is an element of order $2$, then $g$ acts on $A$ as $-1$; thus there is at most one element of order $2$ in $G$.
\end{rem}

The rational group algebra $\QQ[G]$ of $G$ is isomorphic to the sum of simple algebras corresponding to irreducible representations $V$ over $\QQ$:
\[\QQ[G]=\oplus_V \HH_V.\] Define \emph{the rigid group algebra} as the sum over representations without fixed points (WFP): 
\[\QQ[G]^\rig:=\oplus_{V\text{ is WFP}} \HH_V.\] 
%In Section~\ref{rigid_sec} we introduce the rigid quotient $\QQ[G]^\rig$ of the group algebra $\QQ[G]$. 
The following result is an immediate consequence of Theorem~\ref{main_thm}.

\begin{thm}\label{main_cor}
There exists an abelian variety with a rigid action of $G$ in the isogeny class of an abelian variety $A$, if and only if there exists a homomorphism of $\QQ$-algebras $\QQ[G]^\rig\to\End^\circ(A).$
\end{thm}

%We compute it using WEDDERGA package of GAP~\cite{WE}.
In Section~\ref{rigid_surf} we use this result to obtain a classification of finite groups $G$ with a rigid action on an abelian surface. This result can be extracted from the results of Hwang~\cite{Hw}, but we feel that in this particular case it is easier to give an independent proof. Section~\ref{SecSing} is devoted to singularities of $A/G$.

In Section~\ref{main_sec} we heavily use results from the pioneering paper of Katsura on generalized Kummer surfaces over a field of positive characteristic~\cite{Ka}. Let us reformulate some of his results here in a more convenient way.

\begin{thm}~\cite[Theorem 2.4]{Ka}\label{Katsura_quotient}
Assume that $\ch k\neq 2$. A relatively minimal model of $A/G$ is a $K3$ surface if and only if the action of $G$ on $A$ satisfies the following conditions
\begin{itemize}
\item the action of $G$ is rigid;
\item the action of $G$ is symplectic;
\item $A/G$ is singular, and all the singular points of $A/G$ are rational double points.
\end{itemize}
\end{thm}

Katsura classified finite groups $G$ such that there exists a finite field $k$ of characteristic 
$p>5$ and an abelian surface $A$ over $k$ with an action of $G$ such that the quotient is birational to a $K3$ surface.

\noindent

\begin{thm}~\cite[Theorem 3.7]{Ka}\label{Katsura_list}
Let $k$ be an algebraically closed field of characteristic $p>5$.
Let $A$ be an abelian surface over $k$ with an action of a finite group $G$ such that the quotient is birational to a $K3$ surface.
If all elements of $G$ fix the group law of $A$, then $G$ is one of the following groups:
\begin{enumerate}
	\item a cyclic group of order $2,3,4,5,6,8, 10$, or $12$;
	\item a binary dihedral group $Q_{4n}$ of order $4n$, where $2\leq n\leq 6$;
	\item $\SL_2(\FF_3)$, $\ESL_2(\FF_3)$, or $\SL_2(\FF_5)$.
	\end{enumerate}
\end{thm}
 
Katsura showed that any group from Theorem~\ref{Katsura_list} occurs over some finite field of characteristic $p>5$.
As far as the author knows, there is no classification of such groups over a given finite field.
In Section~\ref{main_sec} we prove the following statement.

\begin{thm}\label{main1}
Let $G$ be a finite group, and let $k$ be a perfect field of characteristic $p>2$. 
Assume that $p$ does not divide the order of $G$. 
There exists an abelian surface with a faithful action of $G$ over $k$ such that $A/G$ is birational to a $K3$ surface if and only if $G$ belongs to the Katsura list from Theorem~\ref{Katsura_list}, and one of the following conditions hold:
\begin{enumerate}
	\item  $G$ is $Q_8,Q_{12},\SL_2(\FF_3)$, or a cyclic group of order $2,3,4$, or $6$;
	\item  if $G$ contains a cyclic subgroup of order $n\in\{5,8,12\}$, then 
 $A$ is supersingular, $\FF_{p^2}\subset k$, and $p\not\equiv\pm 1\pmod n$.
	\end{enumerate}
\end{thm}

Finally, in section~\ref{SecZeta} we study traces of Frobenius actions on Neron--Severi groups of supersingular generalized Kummer surfaces. A supersingular $K3$ surface $X_{21}$ over $\FF_p$ with the zeta function $(1^{21},2)$ was constructed by Schuett~\cite{Sch}. 
We give an independent construction on the assumption that $p\equiv 3\pmod 4$.
%in Theorem~\ref{SSZeta2}.
We summarize Theorems~\ref{SSZeta1} and~\ref{SSZeta2} as follows.

\begin{thm}\label{main2}
There exist supersingular $K3$ surfaces over $\FF_{q}$, where $q$ is an even power of $p$,
with the following traces: $22, 14, 10, 6, 4, 2, 0.$

There exist supersingular $K3$ surfaces over $\FF_{q}$, where $q$ is an odd power of $p$,
with the following traces: $20, 18, 14, 10, 8, 6, 2, 0.$
\end{thm}

Note that in all our examples the traces are even. It is natural to ask whether the trace is always even for a supersingular $K3$ surface over $\FF_q$.  

\subsection*{Acknowledgements}
The work was supported by the Israel Science Foundation,  grant No.  1405/22.
I would like to thank Yuri Zarhin for interesting discussions and Matthias Sch\" utt  for email exchanges. I also thank WonTae Hwang for his remarks on the paper. 
I am deeply grateful to the anonymous reviewer for many valuable comments on the paper.
I am a Simons-IUM contest winner, and I am grateful to its sponsors and jury.

\section{Preliminaries}
\subsection{Notation}
\begin{align*}
	\zeta_n: & \text{ a primitive $n$-th root of unity};\\
	\Phi_r: & \text{ the }r\text{-th cyclotomic polynomial};\\
	v_\p:& \text{ a normalized valuation of $\overline{\QQ}$, that is, }v_\p(p)=1\text{, where }\p\text{ is a prime over }p\in\ZZ;\\
	M(r,\HH):&\text{ $r\times r$-square matrices over an algebra }\HH;\\
	\HH_\infty(K):& \text{ the quaternion algebra over }K\text{ with non-trivial invariants at the real places of }K;\\
\HH_p: &\text{ the quaternion algebra over }\QQ\text{ with non-trivial invariants at the real place and at }p;\\
L^\real:&\text{ the totally real subfield of a CM number field }L; \\
Q_{4n}:&\text{ the binary dihedral group of order }4n\text{ (see Section~\ref{rigid_sec})};\\
\ESL_2(k):&\text{ the extended }\SL_2\text{ group over a finite field }k\text{ (see Section~\ref{rigid_surf})};\\
{\QQ[G]}^\rig:&\text{ the rigid quotient algebra of }\QQ[G]. 
\end{align*}

\subsection{Central simple algebras over number fields}
Recall some well-known facts on central simple algebras over number fields~\cite[Section 28]{MO}. Let $L/K$ be an extension of number fields, and let $\HH$ be a simple algebra with center $K$ such that $L\otimes\HH$ is a matrix algebra over $L$, i.e., $\HH$ represents an element of the Brauer group $\Br(L/K)$.
Let $v$ be a prime of $K$. We denote the local invariant of $\HH$ at $v$ by $\inv_v\HH\in\QQ/\ZZ$.
A simple algebra over a number field $K$ is uniquely determined by its local invariants at the places of $K$~\cite[Remark 32.12]{MO}. 

\begin{thm}\label{trivial_inv}
Let $\HH$ be a central simple algebra over $K$, and let $v$ be a prime of $K$.
Let $L$ be an extension of $K$. Choose a prime $w$ of $L$ over $v$.
Then $L\otimes\HH$ is a central simple algebra over $L$ with local invariant $[L_w:K_v]\inv_v\HH$ at $w$.
\end{thm}
\begin{proof}
        It follows from~\cite[Theorem 31.9]{MO}.
\end{proof}
 
Let $K$ be a totally real extension of $\QQ$ of even degree. We denote by $\HH_\infty(K)$ the quaternion algebra with center $K$, and non-trivial invariants exactly at the real places of $K$. 
We denote by $\HH_p$ the quaternion algebra with center $\QQ$, and non-trivial invariants at the real place and at $p$.

\begin{thm}\label{trivial_H}
Let $\HH$ be a simple algebra of dimension $d^2$ over its center $K$. Let $L$ be an extension of $K$ of degree $d$.
Then there is a homomorphism of $L$ to $\HH$ if and only if $L\otimes_K\HH$ is a matrix algebra.
\end{thm}
\begin{proof}
    This follows from~\cite[Theorem 28.5]{MO}.
\end{proof}

\begin{ex} Let $L$ be a CM number field such that the totally real subfield $L^\real$ has even degree. According to the previous theorem, there exists an injection $L\to \HH_\infty(L^\real)$.
\end{ex}

\begin{corollary}\label{Hp_hom}
\begin{enumerate}
\item Let $K$ be a real quadratic extension of $\QQ$.
There is a homomorphism of $\HH_p$ to $\HH_\infty(K)$ if and only if $p$ does not split in $K$.
\item Let $L$ be an imaginary quadratic extension of $\QQ$.
There is a homomorphism of $\HH_p$ to $M(2,L)$ if and only if $p$ does not split in $L$.
\end{enumerate}
\end{corollary}
\begin{proof}
\begin{enumerate}
\item If $p$ does not split in $K$, then, according to Theorem~\ref{trivial_inv}, we have a homomorphism
 \[\HH_p\to\HH_p\otimes K\cong \HH_\infty(K).\]
Assume that $p$ splits in $K$, and that there is a homomorphism $\HH_p\to \HH_\infty(K)$. Then we
get a homomorphism of the simple algebra $\HH_p\otimes\QQ_p$ to the sum of two matrix algebras
\[\HH_\infty(K)\otimes\QQ_p\cong M(2,\QQ_p)\oplus M(2,\QQ_p).\]
A contradiction.
\item There is a homomorphism of $\HH_p$ to $M(2,L)$ if and only if there is a homomorphism in the opposite direction of the centralizers of these algebras in $M(4,\QQ)$, that is, $L\to\HH_p$. 
According to Theorems~\ref{trivial_inv} and~\ref{trivial_H}, the last homomorphism exists if and only if $p$ does not split in $L$.
\end{enumerate}
\end{proof}

\subsection{Abelian varieties over finite fields}
Let $A$ be an abelian variety over $k$.
The ring of $k$-endomorphisms $\End(A)$ of an abelian variety $A$ is finitely generated and torsion-free as a $\ZZ$-module.
Let \[\End^{\circ}(A)=\End(A)\otimes_{\ZZ}\QQ.\]
An element $\phi\in\End(A)$ is called an \emph{isogeny} if $\phi$ is finite and surjective.
For example, multiplication by $m\in\ZZ$ is an isogeny $[m]:A\to A$.
Denote the kernel of $[m]$ by $A[m]$.
This is a finite group scheme over $k$ of order $m^{2\dim A}$~\cite[Remark 7.3]{Milne}. 

An abelian variety $A$ is {\it simple} over $k$ if it does not contain non-trivial abelian subvarieties over $k$. Any abelian variety $A$ over $k$ is $k$-isogenous to a product of simple abelian varieties:
\[A\to \prod_iA_i^{r_i}.\]
This decomposition corresponds to a decomposition of $\End^{\circ}(A)$ into a product of simple algebras
$M(r_i,\End^{\circ}(A_i))$. In particular, $\End^{\circ}(A)$ is a semi-simple $\QQ$-algebra~\cite[IV.19, Corollaries 1 and 2]{Mum}.

Fix a prime number $\ell\neq p$. Let
  \[T_\ell(A) = \plim A[\ell^r](\kk),\text { and } V_\ell(A)=T_\ell(A)\otimes_{\ZZ_\ell}\QQ_\ell\] 
	be the $\ell$-th Tate module of $A$, and the corresponding vector space over $\QQ_\ell$.
 It is known that $T_\ell(A)$ is a free $\ZZ_\ell$-module of rank $2\dim A$.

Let $k$ be a finite field $\FF_q$.
\emph{The Frobenius endomorphism} $F_A: A\to A$ over $k$ is the morphism that is trivial on the topological space and raises functions to their $q$-th powers. 
The endomorphism $F_A$ induces on the Tate module a linear transformation
\[F: T_\ell(A)\to T_\ell(A)\]
such that the action $V_\ell(A)$ is semi-simple.
The characteristic polynomial \[f_A(t) = \det(t-F|V_\ell(A))\]
is called {\it the Weil polynomial of $A$}. It is a monic polynomial of degree $2\dim A$ with rational integer coefficients independent of the choice of $\ell$ (see~\cite[Page 180, Theorem 4]{Mum}, and~\cite[Page 205]{Mum}). 
The following lemma can be proved in the same way as~\cite[IV.2.3]{Milne}.

\begin{lemma}\label{lem_on_Tate_module} 
If $f:A\to B$ is an isogeny, then $T_\ell(f):T_\ell(A)\to T_\ell(B)$ is injective, and the preimage $T=V_\ell(f)^{-1}(T_\ell(B))$ is an 
$F$-invariant submodule of $V_\ell(A)$.
Conversely, if $T\subset V_\ell(A)$ is $F$-invariant $\ZZ_\ell$-submodule of finite rank such that $T_\ell(A)\subset T$, then there exists an abelian variety $B$ defined over $k$ and an $\ell$-isogeny $f:A\to B$ such that $V_\ell(f)$ induces an isomorphism $T\cong T_\ell(B)$.\qed
\end{lemma}

Tate proved that the isogeny class of an abelian variety over $k$ is determined by its characteristic polynomial, that is $f_A(t)=f_B(t)$ 
if and only if $A$ is isogenous to $B$~\cite{Ta66}. Moreover, for any $\ell\neq p$ we have
\[ \End(A)\otimes_{\ZZ}\ZZ_\ell\cong\End_F(T_\ell A),\]
 where $\End_F(T_\ell A)$ is the ring of endomorphisms of $T_\ell(A)$ that commute with $F$.

Assume that $A$ is simple. In this case, $\End^\circ(A)$ is a skew-field, and
the Frobenius endomorphism generates its center $K_A$.
For any embedding $\sigma:K_A\to \CC$ we have  
$|\sigma(F_A)|=\sqrt{q}$~\cite[Page 206, Theorem 4]{Mum}.
Let $L_A$ be the Galois envelope of $K_A$ over $\QQ$, and let $P$ be the set of primes of $L_A$ over $p$.
The slopes of $A$ are defined as the set \[\{\frac{v_\p(F_A)}{v_\p(q)}|\p\in P\},\] where $v_\p$ is the normalized valuation.
We say that $A$ is \emph{supersingular} if all slopes of $A$ are equal to $1/2$. 

\begin{lemma}\label{NoSS}
Let $L$ be a Galois extension of $\QQ$. 
Assume that $\sqrt{q}\not\in\ZZ$, and $p$ is unramified in $L$.
 Then there is no supersingular abelian variety $A$ over $\FF_q$ such that $L_A=L$.
\end{lemma}
\begin{proof}
Let $A$ be an abelian variety over $\FF_q$ such that $L_A=L$. 
Then the slopes of $A$ are the fractions ${v_\p(F_A)}/{v_\p(q)}$ with odd denominator. 
\end{proof}

The abelian variety is \emph{ordinary} if the set of slopes of $A$ is $\{0,1\}$. It is known that $A$ is ordinary if and only if the order of the group $A[p](\kk)$ is equal to $p^{\dim A}$. A simple abelian surface $A$ is called \emph{mixed} if the set of slopes of $A$ is $\{0,1/2,1\}$.

In dimensions $1$ and $2$ Weil polynomials can be explicitly classified. We present this classification in part. For more details, see~\cite{MN} and~\cite{Wa}.

\begin{thm}\cite{Wa}\label{th_curv}
Let $E$ be an elliptic curve over $\FF_q$.
Then $f_E(t)=t^2 - bt + q$, where $|b|\le 2\sqrt{q}$. 
\begin{enumerate}
\item $E$ is ordinary if and only if $v_p(b)=0$. In this case \[\End^\circ(E)\cong\QQ[t]/f_E(t)\QQ[t].\]
%For any $b\in\ZZ$ such that $|b|\le 2\sqrt{q}$ and $v_p(b)=0$ there exists an ordinary elliptic curve $E$ with  $f_E(t)=t^2 - bt + q$.
\item If $E$ is supersingular and $f_E$ is separable, then \[\End^\circ(E)\cong\QQ[t]/f_E(t)\QQ[t],\] and $q$ and $b$ satisfy the following conditions:
\begin{enumerate}
\item $\sqrt{q}\not\in\ZZ$, and $b=0$;
\item $\sqrt{q}\in\ZZ$, $b=0$, and $p\not\equiv 1\pmod 4$;
\item $\sqrt{q}\in\ZZ$, $b=\pm\sqrt{q}$, and $p\not\equiv 1 \pmod 3$;
\item $p=2$ or $p=3$, $\sqrt{q}\not\in\ZZ$, and $b=\pm \sqrt{pq}$.
\end{enumerate}
\item If $E$ is supersingular and $f_E$ is not separable, then $\sqrt{q}\in\ZZ$ is an integer, 
$f_E(t)=(t\pm\sqrt{q})^2$, and $\End^\circ(E)\cong\HH_p$.
\end{enumerate}
In all these cases, there exists an elliptic curve over $\FF_q$ with a given Weil polynomial. 
\end{thm}

\begin{thm}\cite{Wa,Ru,MN}\label{th_surf}
Let $A$ be a simple abelian surface over $\FF_q$, where $q=p^n$.
Then $f_A(t)=P_A(t)^e$, where $P_A\in\ZZ[t]$ is irreducible.  
\begin{enumerate}
\item If $A$ is ordinary or mixed, then $e=1$, and \[\End^\circ(A)\cong\QQ[t]/f_A(t)\QQ[t].\]
\item If $A$ is supersingular, then we have the following possibilities:
\begin{enumerate}
\item $e=1$, and $\End^\circ(A)\cong\QQ[t]/f_A(t)\QQ[t]$. In this case
$f_A(t) = t^4 + a_1 t^3 + a_2 t^2 + a_1 q t + q^2$ is irreducible, and the pair $(a_1,a_2)$ belongs to the following list:
     \begin{enumerate}
     \item $(0,0)$, $n$ is odd, and $p\neq 2$;
     \item $(0,0)$, $n$ is even, and $p\not\equiv 1\pmod 8$;
     \item $(0,q)$, and $n$ is odd;
     \item $(0,-q)$, $n$ is odd, and $p\neq 3$;
     \item $(0,-q)$, $n$ is even, and $p\not\equiv 1\pmod{12}$;
     \item $(\pm\sqrt{q},q)$, $n$ is even, and $p\not\equiv 1\pmod 5$;
    \item $(\pm\sqrt{5q},3q)$, $n$ is odd, and $p=5$;
     \item $(\pm\sqrt{2q},q)$, $n$ is odd, and $p=2$.
     \end{enumerate}
\item $e=2$, $n$ is odd, $P_A(t)=t^2-q$, and $\End^\circ(A)\cong\HH_\infty(\QQ(\sqrt{p}))$;
\item $e=2$, $n$ is even, $P_A(t)=t^2-bt+q$, and one of the following conditions holds
\begin{enumerate}
\item $b=0$, and $p\equiv 1 \pmod 4$;
\item $b=\pm\sqrt{q}$, and $p\equiv 1 \pmod 3$.
\end{enumerate}
In both cases $\End^\circ(A)$ is a quaternion algebra over $\QQ[t]/P_A(t)\QQ[t]$ with non-trivial invariants at finite places over $p$.
\end{enumerate}
\end{enumerate}
In all these cases, there exists a simple abelian surface over $\FF_q$ with a given Weil polynomial. 
\end{thm}

\subsection{\Die modules of abelian varieties}
Let $W(k)$ be the ring of Witt vectors over $k$. We denote by $\sigma: W(k)\to W(k)$ the lift of the Frobenius automorphism of $k$.
The \Die ring $D_k$ is the ring generated over $W(k)$ by formal variables $F$ and $V$ such that $FV=VF=p$, and for any
$a\in W(k)$ we have:
\[ Fa=\sigma(a)F,\text{ and } aV=V\sigma(a). \]
There is an anti-equivalence of categories $M(-)$ between finite commutative group schemes over $k$ of $p$-power order and left $D_k$-modules of finite length~\cite[Theorem 28.3]{Pink}. 

The $p^r$-torsion $A[p^r]$ of an abelian variety $A$ over $k$ is a finite group scheme, and the limit
\[M(A)=\varprojlim_r M(A[p^r])\]
 is called \emph{the \Die module of $A$}. The module $M(A)$ is free over $W(k)$ of rank $2\dim A$.
 
\begin{prop}\cite[Proposition 28.4]{Pink}\label{Pink_prop}
 The cotangent space $T^*_0(A)$ of $A$ at zero is canonically isomorphic to $M(A)/FM(A)$.    
\end{prop}

The localization $W(k)[\frac{1}{p}]$ of $W(k)$ at $p$ is isomorphic to the field of fractions. For a $W(k)$-module $M$ we denote the tensor product 
$M\otimes_{W(k)} W(k)[\frac{1}{p}]$ by $M[\frac{1}{p}]$.
We have the following analog of Lemma~\ref{lem_on_Tate_module}.

\begin{lemma}%~\cite{??}
\label{lem_on_Dmod}
\label{Diedonne_module} If $f:B\to A$ is an isogeny, then $M(f)$ is injective, and
$M=M(f)[\frac{1}{p}]^{-1}(M(B))$ is a \Die submodule of $M(A)[\frac{1}{p}]$. 

Conversely, if $M\subset M(A)[\frac{1}{p}]$ is a \Die submodule of finite 
$W(k)$-rank such that $M(A)\subset M$,
then there exists an abelian variety $B$ over $k$ and a $p$-isogeny $f:B\to A$ such that $M(f)[\frac{1}{p}]$ induces an isomorphism $M(B)\cong M$.\qed
\end{lemma}

Let $r,s\in\ZZ^+$, where $s>0$. We say that a \Die module $M$ is \emph{pure of slope $r/s$}, if there exists a submodule $M'\subset M$ such that 
$M'[\frac{1}{p}]\cong M[\frac{1}{p}]$, and $F^s(M')=p^rM'$.  

\begin{thm}\label{Manin_th}\cite[\S 4]{Manin63}\cite[Theorem 1.3]{DM}
Let $M$ be a \Die module. Then 
\[M\left[\frac{1}{p}\right]\cong \oplus_\lambda M_\lambda\left[\frac{1}{p}\right],\]
where each $M_\lambda\subset M$ is pure of slope $\lambda\in\QQ$.
\end{thm}

We say that the set $\{\lambda\in\QQ|M_\lambda\neq 0\}$ is \emph{the set of slopes of $M$}. 
If $k$ is a finite field and $M=M(A)$ is the \Die module of an abelian variety $A$ over $k$, then by~\cite[Theorem 4.1]{Manin63} the slopes of $M$ and slopes of $A$ coincide. 
This result motivates a more general definition.
An abelian variety $A$ over a perfect field $k$ is \emph{supersingular}, if all slopes of $M(A)$ are equal to $1/2$, and $A$ is \emph{ordinary}, if the set of slopes of $M(A)$ is $\{0,1\}$. 
The following result is due to Tate, Shioda, Deligne, and Oort.

\begin{thm}\cite[Theorem 4.2]{Oort}\label{SSIsogeny}
Let $k$ be a perfect field of characteristic $p$, and let $E$ be a supersingular elliptic curve over a finite field $\FF_q\subset k$.
Then any supersingular abelian surface $A$ over $k$ is isogenous to $E^2$ over $\kk$.
 In particular, $\End^\circ(A)$ is a subalgebra of $M(2,\HH_p)$.
\end{thm}

The structure of the \Die module of an ordinary abelian variety is well known.

\begin{prop}\label{SSDmod}
Let $A$ be an ordinary abelian variety over $k$. 
	Then $M(A)\cong M\oplus M^*$, where $M$ is a \Die module of slope $1$ and rank $\dim(A)$ over $W(k)$, and $T^*_0(A)\cong M/pM.$
\end{prop}
\begin{proof}
According to~\cite[Proposition 15.4]{Pink}, the group scheme $A[p^n]$ is a direct sum of its \'etale part $X^*_n$  and the local part $X_n$. Since $A$ is ordinary, the order of $X^*_n$ is $p^{n\dim A}$; therefore, the order of $X_n$ is also $p^{n\dim A}$.
It follows that \[M=\varprojlim_n M(X_n),\text{ and }M^*=\varprojlim_n M(X^*_n)\] are free $W(k)$-modules of rank $\dim A$, and $M(A)\cong M\oplus M^*$. The Frobenius endomorphism is invertible on the \'etale part $M^*$; therefore, according to Proposition~\ref{Pink_prop}, \[T^*_0(A)\cong M(A)/FM(A)\cong M/pM.\] The proposition is proved.
\end{proof}

\subsection{Finite group actions on abelian varieties}\label{GonAV}
Let $G$ be a finite group, and an action of $G$ on an abelian variety $A$ is given by a homomorphism $G\to\End(A)^*$.
The induced action on $T^*_0(A)$ gives a homomorphism \[G\to \GL_d(k),\] where $d$ is the dimension of $A$.

\begin{lemma}\label{Image_in_SL}
Assume that the action of $G$ on $A$ is faithful. Let $g\in G$ be an element of order $r$. If the image of $g$ in $\GL_d(k)$ is trivial, then $r$ is a power of $p$.
\end{lemma}
\begin{proof}
We may assume that $r\neq p$ is the prime and that the action of $g$ on $T^*_0(A)$ is trivial. 
We may also assume that $g-1$ is an isogeny. Indeed, if $\ker(g-1)$ is not finite, take a quotient $A/B$ by the connected component $B$ of $\ker(g-1)$. Clearly $g$ acts on $A/B$ such that $g-1$ is an isogeny.
According to Proposition~\ref{Pink_prop}, the group scheme corresponding to the \Die module $M(A)/FM(A)$ is a subscheme of $\ker(g-1)$.
We have \[v_p(\deg(g-1))=v_p(\det(M(g)-1));\]
 therefore, $\det(M(g)-1)$ is not a $p$-adic unit.
Since the norm \[N_{\QQ(\zeta_r)/\QQ}(\zeta_r-1)=r,\] is a $p$-adic unit, we get a contradiction with the fact that the eigenvalues of $M(g)$ are roots of unity of order 
$r\neq p$.
\end{proof}

Define a decreasing filtration on $G$ as follows: $G=G_0$, and for $s>0$ $$G_s=\ker(G\to\Aut(M(A)/F^sM(A))).$$
By definition, $G_0/G_1$ is a subgroup of $\GL_d(k)$. 

\begin{lemma} Let $s>0$.
There is an injective homomorphism $G_s/G_{s+1}\to \End(T^*_0A)^+$ to the additive group $\End(T^*_0A)^+$ of  $\End(T^*_0A)$.
In particular, the quotient $G_s/G_{s+1}$ is abelian. 
\end{lemma}
\begin{proof}
If $g$ is an element of $G_s$, then for any $v\in M(A)/F^{s+1}M(A)$ we have $g(v)= v+F^sv_g$ for some unique $v_g\in M(A)/FM(A)$. Moreover, if $v\in FM(A)$, then $v_g=0$.
Since $F$ is injective on $M(A)$, the morphism $F^s$ induces an isomorphism 
\[M(A)/FM(A)\to F^sM(A)/F^{s+1}M(A).\]
Define the morphism \[\alpha_s:G_s\to \Hom(M(A)/FM(A),M(A)/FM(A))\]  as follows: lift $\Bar v\in M(A)/FM(A)$ to some
$v\in M(A)/F^{s+1}M(A)$, and put $\alpha_s(g)(\Bar v)=v_g$.
Clearly, $\alpha_s(g)=0$ if and only if $g(v)= v$ for all $v\in M(A)/F^{s+1}M(A)$, that is, $g\in G_{s+1}$. 
\end{proof}

Assume now that $A$ is a surface. 

\begin{lemma}
Let $Q$ be a $p$-Sylow subgroup of $G$. Then $G_1$ is normal in $Q$, and $Q/G_1$ is annihilated by $p$. 
\end{lemma}
\begin{proof}
The group $G_1$ is normal in $G$; therefore, it is normal in $Q$ as well. 
The quotient $Q/G_1$ is isomorphic to a subgroup of $\GL_2(k)$, and any 
$p$-subgroup of $\GL_2(k)$ is annihilated by $p$. 
\end{proof}

\begin{corollary}
If the $p$-th Sylow subgroup $Q$ of $G$ is not annihilated by $p$, then $G_1$ is not trivial.
\end{corollary}

In fact, the filtration $G_i$ could be non-trivial for small $p$ only.
For example, if the action is symplectic, and the filtration is non-trivial,
it follows from Lemma~\ref{L1} that $p\leq 3$.

\begin{ex}
Let $p=2$, and let $G=\SL_2(\FF_5)$. The $2$-Sylow subgroup of $G$ is $Q_8$; therefore,
 the filtration given by the subgroups $G_i$ is nontrivial.
\end{ex}

\section{Rigid actions on abelian varieties}\label{rigid_sec}
\subsection{An equivalence of categories.}
Let $\C$ be an isogeny class of abelian varieties over $k$, and let $G$ be a finite group.
Denote by $\C_G$ the category of abelian varieties from $\C$ with an action of $G$ that fixes the group law. Recall that such an action on an abelian variety $A$ is given by a homomorphism $G\to\End(A)^*$.
We define the $\Hom$ group as follows: \[\Hom_{\C_G}(A,B)=\Hom^\circ(A,B)^G.\] 

Let $\D_G$ be the following category. Objects of $\D_G$ are group homomorphisms $G\to\End^\circ(A)^*$,  where $A$ in $\C$, and a morphism from $i_A:G\to \End^\circ(A)^*$ to $i_B:G\to \End^\circ(B)^*$ is a $\psi\in\Hom^\circ(A,B)$ 
that fits the diagram:
\[\xymatrix{
&\End^\circ(A)^*\ar@{-->}[dd]_{\psi}\\
G\ar[ur]^{i_A}  \ar[dr]_{i_B}\\
&\End^\circ(B)^*
}\]

%such that for any $g\in G$ \[i_B(g)=\psi\circ i_A(g).\]
There is a natural functor from $\C_G$ to $\D_G$ sending an abelian variety $A$ with a $G$-action to the induced homomorphism $G\to\End^\circ(A)^*$.

\begin{prop}\label{P1}
The natural functor from $\C_G$ to $\D_G$ is an equivalence of categories.
\end{prop}
\begin{proof}
Clearly, the functor is full and faithful. 
We have to prove that the functor is essentially surjective.
Let $A'$ be an abelian variety in $\C$, and let $i_{A'}:G\to \End^\circ(A')^*$ be a homomorphism.
For any $\ell\neq p$ we define a $G$-invariant submodule $T_\ell$ by the formula \[T_\ell=\sum_{g\in G}gT_\ell(A')\subset V_\ell(A').\] 
We claim that $T_\ell=T_\ell(A')$ for almost all $\ell\neq p$. Indeed, 
there exists a natural number $N$ such that $NG\subset \End(A')$;
therefore, if $\ell$ does not divide $Np$, then $G\subset\End(A')\otimes\ZZ_\ell$,
and \[gT_\ell(A')\subset T_\ell(A')\]
for any $g\in G$.
 %Clearly, this inclusion holds for almost all $\ell$.
Now we apply Lemma~\ref{lem_on_Tate_module} to 
the finite number of primes $\ell\neq p$ such that $T_\ell\neq T_\ell(A')$, 
and get an isogeny $A'\to A''$ such that $T_\ell\cong T_\ell(A'')$ for all $\ell\neq p$.
Finally, we put \[M_p=\sum_{g\in G}gM(A')\subset M(A')[\frac{1}{p}].\] According to Lemma~\ref{lem_on_Dmod}, 
there exists an isogeny $A\to A''$ such that
$T_\ell\cong T_\ell(A)$, and $M_p\cong M(A)$.
In particular, \[i_A(G)\subset\End(A)\otimes\ZZ_\ell\] for all $\ell$.
Consider the natural inclusion $\iota:\End(A)\to T_G$ to the lattice generated by $\End(A)$ and the image of $i_A(G)$. For all primes $\ell$ the localization $\iota\otimes\ZZ_\ell$ is an isomorphism;
this implies that $T_G=\End(A)$, and $G\subset\End(A)$. We showed that there exists an abelian variety $A$ in $\C_G$, and the diagram $A'\to A''\leftarrow A$ induces an isomorphism in $\D_G$.
\end{proof}

\begin{corollary}\label{finite_field}
If there exists an abelian variety $A$ with a $G$-action over $k$, then there exists an 
abelian variety $B$ of the same dimension over a finite field $k_1$ endowed with a $G$-action and an injective homomorphism $\End^\circ(A)\to\End^\circ(B)$ such that the $G$-action is given by the composition
\[G\to\End^\circ(A)^*\to\End^\circ(B)^*.\] 

Moreover, if $\FF_{p^2}\not\subset k$, then we can choose $k_1$ such that  
$\FF_{p^2}\not\subset k_1$.

If $A$ is supersingular, then $B$ is also supersingular. 
\end{corollary}
\begin{proof}
Let $k_0$ be the algebraic closure of $\FF_p$ in $k$. 
There exists a finitely generated subfield $k'$ such that both $A$ and the action are defined over $k'$. Choose a smooth connected affine variety $U$ over $k_0$ such that $k'$ is a field of functions on $U$. If we make $U$ smaller, then there exists a smooth abelian scheme $\mathcal{A}$ over $U$ such that $A$ is the general fiber of $\mathcal{A}$.
Let $B$ be a special fiber of $\mathcal{A}$ over a finite extension $k_1$ of $k_0$ of odd degree. Then there exists a natural injection~\cite[Section 1.8.4.1]{LAV}:
\[\End(A)\to\End(B).\]
Therefore, $G$ acts on $B$. 

The Newton polygon of $A$ is not higher than the Newton polygon of $B$ according to the Grothen\-dieck--Katz specialization theorem~\cite{Ka78}; therefore, if $A$ is supersingular, then $B$ is also supersingular. 

If $k_0$ is finite, we are done. If $k_0$ is an algebraic closure of $\FF_p$, then 
both $B$ and the action are defined over some finite subfield.
\end{proof}

\begin{prop}
The action of $G$ on $A$ is rigid if and only if any non-trivial $g\in G$ has only finite number of fixed points.
\end{prop}
\begin{proof}
The eigenvalues of the action of $g\in G$ on $V_\ell(A)$ are primitive roots of unity if and only if $\ker(1-V_\ell(g)^r)\subset V_\ell(A)$ is trivial for any proper divisor $r$ of the order of $g$.
Let $B$ be the connected component of identity in $\ker(1-g^r)$. Then $B$ is an abelian subvariety, and $g^r$ has only a finite number of fixed points if and only if $\dim B=0$, and if and only if $V_\ell(B)=0$.
We have to prove that $V_\ell(B)=\ker(V_\ell(1-g^r))$. 
Clearly, $V_\ell(B)\subset\ker(V_\ell(1-g^r))$. 
We will prove that any $v\in \ker(V_\ell(1-g^r))$ belongs to $V_\ell(B)$.
Since $V_\ell(A)= T_\ell(A)\otimes\QQ_\ell$, we have $v=v'\otimes\frac{1}{\ell^a}$ for some natural number $a$. The vector $v'$ is represented by a sequence $v_j\in A[\ell^j]$ such that $v_j=\ell v_{j+1}$, and $g^r(v_j)=v_j$. Since the quotient $\ker(1-g^r)/B$ is finite, $v_j\in B$ for large $j$; therefore, $v\in V_\ell(B)$.
\end{proof}

\begin{corollary}
If $A$ is simple over $k$, then any faithful action of a finite group on $A$ is rigid.
\end{corollary}

Denote by $\C_G^\rig$ the full subcategory in $\C_G$ of abelian varieties with a rigid action of $G$.
The homomorphism $i_A: G\to\End^\circ(A)^*$ induces a representation on $V_\ell(A)$.
We say that $i_A$ is \emph{rigid} if this representation is without fixed points.
Let $\D_G^\rig$ be the full subcategory of rigid objects in $\D_G$.
 According to Proposition~\ref{P1}, $\C_G^\rig$ and $\D_G^\rig$ are equivalent. 

Let $\HH$ be a semisimple algebra over $\QQ$. Denote by $\C_\HH$ the following category. Objects of $\C_\HH$ are homomorphisms of $\QQ$-algebras $\HH\to\End^\circ(A)$, where $A$ is in $\C$, and 
a morphism from $h_A:\HH\to \End^\circ(A)$ to $h_B:\HH\to \End^\circ(B)$
is an element $\psi\in\Hom^\circ(A,B)$ that fits the diagram:
%such that for any $h\in \HH$ 
\[\xymatrix{
&\End^\circ(A)\ar@{-->}[dd]_{\psi}\\
\HH\ar[ur]^{h_A}  \ar[dr]_{h_B}\\
&\End^\circ(B)
}\]
%$$h_B(h)=\psi\circ h_A(h).$$

%Recall that the rigid group algebra $\QQ[G]^\rig$ is the sum over representations without fixed points.
A homomorphism $i_A$ from $G$ to $\End^\circ(A)^*$ induces the natural homomorphism of $\QQ$-algebras $h_A:\QQ[G]\to\End^\circ(A)$. 
If $i_A$ is rigid, then $h_A$ is the composition of the natural projection $\QQ[G]\to\QQ[G]^\rig$ and a homomorphism \[h_A^\rig:\QQ[G]^\rig\to\End^\circ(A).\]
For $\HH=\QQ[G]^\rig$ this gives an equivalence of $\D_G^\rig$ and $\C_\HH$. 
We proved the main result of this section.

\begin{thm}\label{main_thm}
Let $G$ be a group with a representation without fixed points, and let $\HH=\QQ[G]^\rig$.
Then the categories $\C_G^\rig$ and $\C_\HH$ are equivalent.\qed
\end{thm}

We now consider the case of a cyclic group $G$.

\begin{lemma}\label{eigenvalues}
Let $G=C_n$ be a cyclic group of order $n$. 
    \begin{enumerate}
        \item We have $\QQ[G]^\rig\cong \QQ(\zeta_n)$.
        \item Let $g$ be a generator of $G$. If $A$ is an abelian variety with a rigid action of $G$, then all primitive $n$-th roots of unity are eigenvalues of the action of $g$ on $M(A)[\frac{1}{p}]$ and on $V_\ell(A)$ for all $\ell\neq p$.
    \end{enumerate}
\end{lemma}
\begin{proof}
    Part $(1)$ follows from the isomorphism of algebras 
    \[\QQ[G]\cong \oplus_{d|n}\QQ(\zeta_d).\]
    According to Theorem~\ref{main_thm}, if the action of $G$ on $A$ is rigid, then
    $\QQ(\zeta_n)$ is a subalgebra of $\End^\circ(A)$.
    By~\cite[Theorem 2.1.1]{Ri76}, $V_\ell(A)$ is a free module over $\QQ(\zeta_n)\otimes\QQ_\ell$, where $\zeta_n$ acts as $g$. Hence, $V_\ell(A)\otimes\Bar\QQ_\ell$ is a free module over 
    \[\QQ(\zeta_n)\otimes\Bar\QQ_\ell\cong\oplus_{\phi_\zeta}\Bar\QQ_\ell,\] 
    where the sum is over all immersions $\phi_\zeta:\QQ(\zeta_n)\to\Bar\QQ_\ell$.
    Any such immersion is uniquely determined by the primitive $n$-th root of unity: $\zeta=\phi_\zeta(\zeta_n)$. It follows that
    \[V_\ell(A)\otimes\Bar\QQ_\ell\cong\oplus_\zeta V_\zeta,\] where $g$ acts on $V_\zeta$ with eigenvalue $\zeta$. The same argument can be applied to the \Die module $M(A)$.    
    The lemma is proved.
\end{proof}

\begin{corollary}\label{cor2}
Let $G$ be a cyclic group of order $n$, and let $\HH=\QQ(\zeta_n)$ be the cyclotomic field.
Then there is an equivalence of categories $\C_G^\rig$ and $\C_\HH$.
\end{corollary}

\begin{rem}
 The equivalences of categories $\C_G^\rig$ and $\C_\HH$ from the previous corollary are parametrized by identifications of $G$ and the group of $n$-th roots of unity.
\end{rem}

%\begin{corollary}\label{main_cor}
%There exists an abelian variety with a rigid action of $G$ in the isogeny class of an abelian %variety $A$, if and only if there exists a homomorphism of $\QQ$-algebras %$\QQ[G]^\rig\to\End^\circ(A).$
%\end{corollary}

\begin{corollary}\label{GisSLsubgroup}
If there exists a supersingular abelian surface $A$ with a rigid action of $G$ over a perfect field $k$, then there exists a supersingular abelian surface with a rigid action of $G$ over $\FF_{p^2}$. In particular, $G/G_1$ is a subgroup of $\GL_2(\FF_{p^2})$.
\end{corollary}
\begin{proof}
 Let $E$ be a supersingular elliptic curve over $\FF_{p^2}$ such that $\End^\circ(E)\cong\HH_p$.
According to Theorem~\ref{SSIsogeny}, $\End^\circ(A)$ is a subalgebra of $\End^\circ(E^2)$. 
 We get a homomorphism $G\to \End^\circ(E^2)^*$. 
By Theorem~\ref{main_thm}, there exists an isogeny from $E^2$ to an abelian surface over $\FF_{p^2}$ with a rigid action of $G$.
\end{proof}

\subsection{Finite subgroups of endomorphism algebras.}
The binary dihedral group $Q_{4n}$ of order $4n$ is the group generated by two elements $a$ and $b$ with the following relations:
$$a^n=b^2,\quad b^4=1,\quad aba=b.$$ 

Let $4n=2^rm$, where $m$ is odd. For any divisor $d$ of $m$ define a representation $\HH(n,d)$ of $Q_{4n}$ as follows.
The element $a^{2n/d}$ generates a subgroup of order $d$, and the quotient of $Q_{4n}$ by this subgroup is isomorphic to $Q_{4n/d}$.
Let $\HH_\RR$ be the usual real quaternion algebra generated by $i$ and $j$ over its center $\RR$. 
We generate a subring $\HH(n,d)$ of $\HH_\RR$ by $\QQ({\zeta_{2n/d}})\subset\CC\cong\RR[i]$, and $j$.
 In this representation, $a$ acts as a multiplication of $\zeta_{2n/d}$, and $b$ acts as $j$.

\begin{lemma} \label{Q_rig}
We have $\QQ[Q_{4n}]^\rig\cong\HH(n,1)$.
In particular, $\QQ[Q_8]^\rig\cong\HH_2$, and $\QQ[Q_{12}]^\rig\cong\HH_3$. 
\end{lemma}
\begin{proof}
The quotient of $Q_{4n}$ by the subgroup generated by $b^2$ is the dihedral group $D_{n}$ of order $2n$.
We find that the group algebra $\QQ[Q_{4n}]$ has the following decomposition:
\[\QQ[Q_{4n}]=\oplus_{d|m} \HH(n,d)\oplus \QQ[D_{n}],\] due to the formula
\[\dim\QQ[Q_{4n}]=4n=2n+2\sum_{d|m}\phi(2n/d)=\dim\QQ[D_{n}]+\sum_{d|m}\dim\HH(n,d).\]
The only summand corresponding to a faithful representation is $\HH(n,1)$. 
This representation is rigid because $b^2$ acts as $-1$, and $a$ acts as $\zeta_{2n}$. 
The lemma is proved.
\end{proof}

We need the following classic result of Burnside (see~\cite[Section 5.3]{Wolf}). 

\begin{thm}\cite{Bu}
Let $G$ be a group with a representation without fixed points. Then all odd Sylow subgroups of $G$ are cyclic, and the even Sylow subgroup is cyclic or isomorphic to $Q_{2^r}$.
\end{thm}

We say that a cyclic group of order $n$ is \emph{small} if $n\in\{1,2,3,4,5,6,8,10,12\}$.

\begin{corollary}\label{c240}
Let $A$ be an abelian surface with a rigid action by a finite group $G$. Then the order of $G$ divides $240$, and any cyclic subgroup of $G$ is small.
\end{corollary}
\begin{proof}
A cyclic subgroup of $G$ generates a commutative subalgebra of $\End^\circ(A)$ of dimension not greater than $4$; therefore, any cyclic subgroup is small.
Moreover, since for an odd prime $p$ the $p$-Sylow subgroup $G^{(p)}$ is cyclic,  
we have $p\leq 5$, and $G^{(p)}\cong C_p$.
We showed that the order of $G$ is $2^rm$, where $m$ divides $15$.
If the $2$-Sylow subgroup is commutative, then, as before, $r\leq 3$.
Otherwise, the subalgebra generated by the $2$-Sylow subgroup of $G$ contains $\QQ[Q_{2^r}]^\rig$ as a direct summand. According to Lemma~\ref{Q_rig}, $r\leq 4$.
\end{proof}

\section{Rigid and symplectic actions on abelian surfaces}\label{rigid_surf}
In this section we classify finite groups with a rigid and symplectic action on an abelian surface over $k$.
Let $r$ be a power of an odd prime $p$. Denote by $\ESL_2(\FF_r)$ the subgroup of
$\SL_2(\FF_{r^2})$ generated by $\SL_2(\FF_r)$ and an element given by the diagonal matrix 
\[\left (\begin{array}{c c}
x & 0\\
0 & x^{-1}\\	
\end{array}\right),\] where $x\in\FF_{r^2}\setminus \FF_r$, and $x^2$ generates $\FF_r^*$.

\begin{rem}
The group $\ESL_2(\FF_3)$ is the binary octahedral group.
\end{rem}

\begin{thm}[Dickson]\cite[Chapter 3, Theorem 6.17]{FG}\label{Dickson}
Let $G$ be a finite subgroup of $\SL_2(\kk)$. If the order of $G$ is relatively prime to $p$, then $G$ is isomorphic to one of the following groups:
\begin{enumerate}
	\item a cyclic group $C_n$;
	\item a binary dihedral group $Q_{4n}$;
	\item $\SL_2(\FF_3)$, $\ESL_2(\FF_3)$, or $\SL_2(\FF_5)$.
\end{enumerate}
 
Suppose that $p$ divides the order of $G$, and $Q$ is a Sylow $p$--subgroup of $G$. Then one of the following cases occurs:

\begin{enumerate}
\item $G/Q$ is cyclic, and $Q$ is commutative and annihilated by $p$;
	\item $G\cong\SL_2(\FF_q)$, where $q$ is a power of $p$;
	\item $p=2$, and $G$ is a dihedral group $D_{n}$ of order $2n$;
	\item $p=3$, and $G\cong\SL_2(\FF_5)$;
	\item $\ESL_2(\FF_q)$, where $q$ is a power of $p$.
\end{enumerate}
\end{thm}

We now state the main classification result of the section.

\begin{thm}\label{rigid_actions}
Let $G$ be a finite group with a rigid and symplectic action on an abelian surface over $k$.
%Assume that $p>2$. 
Then $G$ is one of the following groups.

\begin{enumerate}
\item $G$ is a small cyclic group of order $n$, and $\QQ[G]^\rig\cong \QQ(\zeta_n)$;
\item $G$ is a binary dihedral group $Q_{4n}$ of order $4n$, where $2\leq n\leq 6$ 
%the algebra $\QQ[G]^\rig$ is 

\begin{center}
\begin{longtable}{|c|c|c|c|c|c|}
\hline 
$G$  & $Q_8$ & $Q_{12}$ & $Q_{16}$                &  $Q_{20}$                  & $Q_{24}$ \\
\hline 
$\QQ[G]^\rig$ &  $\HH_2$  & $\HH_3$ & $\HH_\infty(\QQ(\sqrt{2}))$ & $\HH_\infty(\QQ(\sqrt{5}))$ & $\HH_\infty(\QQ(\sqrt{3}))$ \\
\hline
\end{longtable}
\end{center}

\item $G$ is $\SL_2(\FF_3)$, $\CSU_2(\FF_3)$, or $\SL_2(\FF_5)$

\begin{center}
\begin{longtable}{|c|c|c|c|}
\hline 
$G$  & $\SL_2(\FF_3)$             & $\CSU_2(\FF_3)$ & $\SL_2(\FF_5)$\\
\hline 
$\QQ[G]^\rig$ & $\HH_2$  &$\HH_\infty(\QQ(\sqrt{2}))$ & $\HH_\infty(\QQ(\sqrt{5}))$\\
\hline
\end{longtable}
\end{center}

\item $p=5$, and $G$ is $\CSU_2(\FF_5)$, or $C_5\rtimes C_8$

\begin{center}
\begin{longtable}{|c|c|c|c|}
\hline 
$G$  &             $\CSU_2(\FF_5)$ & $C_5\rtimes C_8$\\
\hline 
$\QQ[G]^\rig$ & $M(2,\HH_5)$  & $M(2,\HH_5)$  \\
\hline
\end{longtable}
\end{center}

\item $p=3$, and $G$ is $C_3\rtimes C_8$, $C_3\times Q_8$, or $C_3\rtimes Q_{16}$

\begin{center}
\begin{longtable}{|c|c|c|c|}
\hline 
$G$  & $C_3\rtimes C_8$ & $C_3\times Q_8$ &  $C_3\rtimes Q_{16}$\\
\hline 
$\QQ[G]^\rig$  & $M(2,\QQ[\zeta_4])$ & $M(2,\QQ[\zeta_3])$ &  $M(2,\HH_3)$  \\
\hline
\end{longtable}
\end{center}

\item $p=2$, and $G$ is $C_3\rtimes C_8$, or $C_3\times Q_8$

\begin{center}
\begin{longtable}{|c|c|c|}
\hline 
$G$  &  $C_3\rtimes C_8$ & $C_3\times Q_8$   \\
\hline 
$\QQ[G]^\rig$ & $M(2,\QQ[\zeta_4])$ & $M(2,\QQ[\zeta_3])$  \\
\hline
\end{longtable}
\end{center}
\end{enumerate}
\end{thm}

\begin{rem}
In cases $(4),(5)$, and $(6)$ the algebra $\QQ[G]^\rig$ is a subalgebra of $M(2,\HH_p)$.
According to Theorems~\ref{trivial_H} and~\ref{main_cor},
there is a rigid action of the group $G$ on some supersingular abelian surface.
In Section~\ref{main_sec} we study rigid and symplectic actions in other cases.
\end{rem}

\begin{lemma}\label{L1} 
If $p\geq 5$, then $G$ is a subgroup of $\SL_2(k)$.
\end{lemma}
\begin{proof}
According to Lemma~\ref{Image_in_SL}, any element of order prime to $p$ has a non-trivial image in $\SL_2(k)$. 
We have to prove the lemma for an element $g$ of order $5$ and $p=5$. The action of $g$ on the \Die module $M(A)$, induces an action of $\ZZ[\zeta_5]$ on $M(A)$.
Note that $\QQ_5[\zeta_5]$ is totally ramified over $\QQ_5$; thus $M(A)$ is a free module of rank $1$ over the DVR $W(k)\otimes\ZZ_5[\zeta_5]$, and $1-g$ acts as a uniformiser.
It follows that the action of $1-g$ on $M(A)/5M(A)$ is nilpotent with one-dimensional kernel.
The cotangent space $T_0^*(A)$ is a two-dimensional quotient of $M(A)/5M(A)$; therefore, the action of $g$ on $T_0^*(A)$ is non-trivial.
\end{proof}

We now recall some basic facts on group extensions~\cite{FG}. Let $C$ be a normal subgroup of a group $E$.
 We say that $E$ is an extension of $C$ by $S=E/C$. If $C$ is abelian, then there is a natural action of $S$ on $C$, and equivalence classes of extensions with such an action correspond to elements of $H^2(S,C)$. In particular, if $H^2(S,C)$ is trivial, then
$E\cong C\rtimes S$ is a semidirect product.

\begin{lemma}\label{L2} 
Assume that $G_1$ is abelian and that $\Bar G=G/G_1$ has an abelian normal subgroup $H$ of prime to $p$ order.
Then $H^2(\Bar G, G_1)=H^2(\Bar G/H, G_1^H)$. 
\end{lemma}
\begin{proof}
Since $G_1$ is a $p$-group, $H^1(H,G_1)=H^2(H,G_1)=0$.
The lemma follows from the inflation-restriction exact sequence:
\[ 0\to H^2(\Bar G/H, G_1^H)\to H^2(\Bar G, G_1)\to H^2(H, G_1).\qedhere\]
\end{proof}

\begin{lemma}\label{L3}
The following groups do not have rigid representations: $D_3$, $D_5$, $C_5\rtimes C_4$, $D_3\times Q_8$, $C_4.A_4$, and $C_4.A_5$.
%\item The following groups have a cyclic subgroup of order grater than $12$: $Q_{40}$, $Q_{48}$, $Q_{60}$, $Q_{80}$, 
 %$C_3\times Q_{16}$, $C_3\times Q_{20}$, $C_5\times Q_{8}$, $C_5\times Q_{16}$, $C_5\rtimes Q_{16}$.
%\end{enumerate}
\end{lemma}
\begin{proof}
Each of the groups $D_3$, $D_5$, and $C_5\rtimes C_4$ has only one irreducible faithful representation over $\QQ$, but the trace of the element of order $2$ is zero; therefore, the representations are not rigid. The groups $D_3\times Q_8$, $C_4.A_4$, and $C_4.A_5$ contain a subgroup $C_2^2$, and do not have rigid representations according to Remark~\ref{Rem1}.% The second part of the lemma is straightforward.
\end{proof}

\begin{proof}[Proof of Theorem~\ref{rigid_actions}]
According to Lemma~\ref{c240}, the order of $G$ divides $240$, and any cyclic subgroup is small.
If $p> 5$, Lemma~\ref{L1} and the Dickson Theorem~\ref{Dickson} give first three cases. 
If $p=5$, then according to Lemma~\ref{L1} and Dickson Theorem~\ref{Dickson} we get that $G$ could be a group from cases $(1-3)$, $\CSU_2(\FF_5)$, or an extension of $C_5$ by a cyclic group of order dividing $48$. It is straightforward to check that such extensions without non-small cyclic subgroups are either cyclic or belong to the following list: $D_5=C_5\rtimes C_2$, $Q_{20}$, $C_5\rtimes C_4$, and $C_5\rtimes C_8$.  The groups $D_5$, and $C_5\rtimes C_4$ have no rigid representations according to Lemma~\ref{L3}.

Let $p=3$. Let $G_1$ be a normal $p$-subgroup of $G$ from Section~\ref{GonAV}; in particular, $G/G_1$ is isomorphic to a subgroup of $\SL_2(k)$.
We have two cases: either $G_1$ is trivial, or $G_1\cong C_3$. In the first case it follows from the Dickson Theorem that $G$ is either a group from cases $(1)-(3)$,
or is an extension of $C_3$ by a cyclic group of order dividing $80$. 
Since such an extension does not contain $C_{15}$, its order divides $48$.
As before, we get the following list of non-cyclic extensions:  $D_3=C_3\rtimes C_2$, $Q_{12}=C_3\rtimes C_4$, and $C_3\rtimes C_8$. 
The group $D_3$ has no rigid representations according to Lemma~\ref{L3}.

In the second case, we have to consider extensions of $G_1\cong C_3$ by non-cyclic subgroups of $\SL_2(k)$ of order dividing $80$, i.e, $Q_8$, $Q_{16}$, $Q_{20}$, $Q_{40}$, or $Q_{80}$.
Since $3$ is coprime to the orders of these groups, $H^2(G/G_1,G_1)$ is trivial; therefore, any extension is a semidirect product. The groups without non-small cyclic subgroups are 
$Q_{24}\cong C_3\rtimes Q_{8}$, $C_3\times Q_{8}$, and $C_3\rtimes Q_{16}$.

Let $p=2$. The quotient group $\Bar G=G/G_1$ is a subgroup of $\SL_2(k)$ without non-small cyclic subgroups; thus we have the following possibilities for $\Bar G$: 
\[C_2,C_3,C_5, C_6, C_{10}, C_2\times C_2,C_6\times C_2,C_{10}\times C_2, S_3=D_3=\SL_2(\FF_2), D_5, A_4=C_2^2\rtimes C_3, A_5=\SL_2(\FF_4).\]
%where $A_4$ appears as an extension of $C_2\times C_2$ by $C_3$.
Assume first that $G_1$ is cyclic. If $\Bar G$ is abelian, $D_3$, or $D_5$, then 
there is a normal subgroup $H$ of odd order such that $\Bar G/H$ is a $2$-group. 
According to Lemma~\ref{L2}, any extension of $G_1$ by $\Bar G$ is uniquely determined by the
 $2$-Sylow subgroup $Q$ of this extension.
 Since $Q$ is either cyclic or isomorphic to $Q_8$ or $Q_{16}$ we get the following groups without non-small cyclic subgroups:
\begin{enumerate}
\item $C_4$, $C_8$, or $Q_{16}$, if $\Bar G=C_2$;
\item $C_6$, or $C_{12}$, if $\Bar G=C_3$;
\item $C_{10}$, if $\Bar G=C_5$;
\item $C_{12}$, if $\Bar G=C_6$;
\item no groups, if $\Bar G=C_{10}$;
\item $Q_8$, or $Q_{16}$, if $\Bar G=C_2^2$;
\item $C_3\times Q_8$, if $\Bar G=C_6\times C_2$;
\item no groups, if $\Bar G=C_{10}\times C_2$;
\item $Q_{12}$, $Q_{24}$, or $C_3\rtimes C_8$, if $\Bar G=D_3$, 
\item $Q_{20}$, if $\Bar G=D_5$.
\end{enumerate}

If $\Bar G$ is $A_4$ or $A_5$, then $G_1$ is $C_2$ or $C_4$. 
The only non-trivial extensions of $A_4$ and $A_5$ by $C_2$ are $\SL_2(\FF_3)$ and $\SL_2(\FF_5)$, and by $C_4$ are $C_4.A_4$ and $C_4.A_5$ respectively.
The groups $C_4.A_4$ and $C_4.A_5$ have no rigid representations according to Lemma~\ref{L3}.

If $G_1=Q_{16}$, then $\Bar G$ is $C_3$ or $C_5$, and any such extension is trivial, because $\Aut(Q_{16})$ is a $2$-group; in this way we get two groups with non-small cyclic subgroups:
$C_3\times Q_{16}$, and $C_5\times Q_{16}$.
If $G_1=Q_{8}$, and $\Bar G$ is $C_3$ or $C_5$, then either extension is trivial, or $G=Q_8\rtimes C_3\cong\SL_2(\FF_3)$, because $\Aut(Q_{8})\cong S_4$.

Finally, assume that $G_1$ is $Q_8$, and the $2$-Sylow subgroup of $G$ is $Q_{16}$.
If $\Bar G$ is cyclic of even order, then any homomorphism from $\Bar G$ to $\Aut(Q_{8})\cong S_4$ factors through $C_2$; therefore, either $G$ is $Q_{16}$, or contains a non-small cyclic subgroup.
If $\Bar G=C_2\times C_2$, then the order of the $2$-Sylow subgroup is greater than $16$.
If $\Bar G$ is $D_3$ or $D_5$, then any extension that contains $Q_{16}$ has a non-small cyclic subgroup.

We use the WEDDERGA package of GAP~\cite{WE} to compute $\QQ$-algebras for suitable representations without fixed points. 
\end{proof}

\section{Singularities of $A/G$}\label{SecSing}
In this section we study resolutions of singularities of the quotients $A/G$.
%and compute the Galois action on the exceptional graph.
The next proposition is a generalization of the table from Katsura's paper~\cite{Ka} for cyclic groups and of results due to Fujiki~\cite{Fu} for $Q_8$, $Q_{12}$, and $\SL_2(\FF_3)$. 
%We sketch a proof for the sake of completeness.

\begin{prop}\label{Sing} 
Let $A$ be an abelian surface with a rigid and symplectic action of a group $G$ over $k$. 
Assume that $p$ does not divide the order of $G$. Then the singularities of the quotient surface $A/G$, and the rank of the Neron--Severi group of the minimal resolution of singularities $X$ are as follows.

\begin{center}
\begin{longtable}{|c|c|c|}
\hline 
$G$  & Singularities  &  $\rk\NS(X)$  \\
\hline 
$C_2$ & $16A_1$ & $\geq 17$ \\    
$C_3$ & $9A_2$ & $\geq 19$ \\    
$C_4$  &  $4A_3+6A_1$ &  $\geq 19$ \\
$C_5$ & $5A_4$  & $22$ \\
$C_6$ & $A_5+4A_2+5A_1$ & $\geq 19$ \\    
$C_8$ & $2A_7+A_3+3A_1$ & $22$ \\    
$C_{10}$ & $A_9+2A_4+3A_1$ & $22$\\
$C_{12}$ & $A_{11}+A_3+2A_2+2A_1$ & $22$ \\
$Q_8$  & $2D_4+3A_3+2A_1$  & $\geq 20$      \\  
         &  $4D_4+3A_1$ &  $\geq 20$        \\  
$Q_{12}$ & $D_5+3A_3+2A_2+A_1$ & $\geq 20$ \\
$Q_{16}$ & $2D_6+D_4+A_3+A_1$ & $22$ \\
$Q_{20}$ & $D_7+A_4+3A_3$ & $22$ \\
$Q_{24}$ & $D_8+D_4+2A_3+A_2$ & $22$ \\
$\SL_2(\FF_3)$  &  $E_6+D_4+4A_2+A_1$  &  $\geq 20$      \\ 
$\ESL_2(\FF_3)$  &  $E_7+D_6+A_3+2A_2$  &  $22$      \\ 
$\SL_2(\FF_5)$  &  $E_8+D_4+A_4+2A_2$  &  $ 22$      \\ 
\hline
\end{longtable}
\end{center}
\end{prop}

First, we prove a more general lemma. Let $H$ be a subgroup of $G$.
Denote the number of points with the stabilizer isomorphic to $H$ by $N(H)$. 

\begin{lemma}\label{sing_lemma}
Let $G$ be a finite group of order $n$ with a rigid action on an abelian variety $A$.
Assume that $p$ does not divide the order of $G$.
\begin{enumerate}
        \item If $x\neq 0$ is a fixed point of $G$, then there exists a prime $\ell$ 
        such that $x\in A[\ell](\kk)$, and $n=\ell^r$ is a power of $\ell$.
        \item If $G$ is cyclic and $n=\ell^r$, then the set of fixed points of $G$ 
        is a subgroup of $A[\ell](\kk)$ of order $\ell^{r_n}$, where \[r_n=\frac{2\dim A}{(\ell-1)\ell^{r-1}}\].
        \item If $H$ is a subgroup of $G$ such that $\ell$ does not divide the order of $H$, then the action of $H$ on $A[\ell](\kk)\setminus\{0\}$ is free.
        \item Let $G^{(\ell)}$ be the $\ell$-th Sylow subgroup of $G$, and let $s_\ell$ be the number of $\ell$-th Sylow subgroups. If $N_\ell$ is the number of points with stabilizer equal to $G^{(\ell)}$, then $N(G^{(\ell)})=s_\ell N_\ell$.
         \item Let $G=C_4$, and let $A$ be a surface. Then $N(C_4)=4$, and there are $6$ orbits of length $2$.
       \end{enumerate}
\end{lemma}
\begin{proof}
Let $g\in G$ be an element of order $d$. Since the action is rigid, the Tate module $T_\ell(A)$ is a free module over the regular local ring $\ZZ_\ell[\zeta_d]$, where $\zeta_d$ acts as $g$.
If $d=\ell$ is a prime and $x$ is a fixed point of $G$, then $(\zeta_\ell-1)x=0$.
Therefore, the norm \[N_{\QQ(\zeta_\ell)/\QQ}(\zeta_\ell-1)=\ell\] also annihilates $x$.
This proves $(1)$.
%If $\ell$ does not divide $d$, then $g-1$ is invertible on $T_\ell(A)$, and there is no fixed points of order $d$. Suppose that $\ell$ divides $d$. Since the fixed points of $g$ are also fixed by $g^\ell$ the same argument shows that there are no $\ell$-torsion points fixed by $g$, if $d$ is not a power of $\ell$. 

Let $G$ be cyclic of order $n=\ell^r$, and let $g\in G$ be a generator. 
Then $T_\ell(A)$ is a free module over $\ZZ_\ell[\zeta_n]$ of rank $r_n$.
The set of fixed points of $g$ is isomorphic to 
\[T_\ell(A)/(g-1)T_\ell(A)\cong (\ZZ_\ell[\zeta_n]/(\zeta_n-1)\ZZ_\ell[\zeta_n])^{r_n}\cong\FF_\ell^{r_n}.\]

We proved $(2)$. The parts $(3)$ and $(4)$ are easy.

We now prove $(5)$. Clearly, $A[2](\kk)$ is fixed by $C_2\subset C_4$; therefore, according to $(2)$, we have $N(C_4)=4$, and $12$ points form $6$ orbits of length $2$.
\end{proof}

\begin{proof}[Proof of Proposition~\ref{Sing}]
We use the notation of Lemma~\ref{sing_lemma}.
The table for cyclic groups can be easily computed using $(1), (2)$, and $(3)$ of Lemma~\ref{sing_lemma}.

Let $G=Q_8$, and let $H=C_4\subset G$ be a normal subgroup. 
According to Lemma~\ref{sing_lemma}.(1), we have to compute the orbits of $G$ on $A[2](\kk)$. 
The group $H$ has $4$ fixed points by Lemma~\ref{sing_lemma}.(5).
Since the length of any non-trivial orbit is even, and the origin is fixed by $G$, 
either $N(G)=4$, or $N(G)=2$.
Note that $Q_8$ contains $3$ cyclic subgroups of order $4$, and if $2$ such subgroups fix a point, then the point is fixed by $G$. In other words, either $N(C_4)=0$, or $N(C_4)=6$.
In the first case, we have $N(C_2)=12$, and in the second case $N(C_2)=8$. 

Let $G=Q_{16}$, and let $H=C_8\subset G$ be the only cyclic subgroup of order $8$. 
The action of $G/H$ on $A[2](\kk)/H$ is very simple, and we have $N(G)=N(Q_8)=2$, and $N(C_8)=0$. Moreover, since there is a second conjugacy class of $C_4$ in $G$ that is not contained in $C_8$, we have $N(C_4)=4$, and $N(C_2)=8$.

From the previous calculations we can easily compute the orbits in many cases
using $(3)$ and $(4)$ of Lemma~\ref{sing_lemma} and the following information: 
\begin{itemize}
    \item $Q_{12}$ has $3$ cyclic $4$-Sylow subgroups, thus $N(G)=1$, $N(C_4)=9$, $N(C_3)=8$, and $N(C_2)=6$;
    \item $Q_{20}$ has $5$ cyclic $4$-Sylow subgroups, thus  $N(G)=1$, $N(C_5)=4$, and $N(C_4)=15$;
    \item $\SL_2(\FF_3)$ has $4$ cyclic $3$-Sylow subgroups, and a normal subgroup isomorphic to $Q_8$, thus $N(G)=1$, $N(Q_8)=3$, $N(C_3)=32$, and $N(C_2)=12$;
    \item $\SL_2(\FF_5)$ has $10$ cyclic $3$-Sylow subgroups, $6$ cyclic $5$-Sylow subgroups, and $5$ Sylow subgroups isomorphic to $Q_8$, thus $N(G)=1$, $N(Q_8)=15$, $N(C_5)=24$, and $N(C_3)=80$.
\end{itemize}

Let $G=Q_{24}$. The action of $C_2\cong G/Q_{12}$ on $A[2](\kk)/Q_{12}$ clearly fixes the image of the origin and the image of a point with the stabilizer $C_2$; therefore, $N(G)=1$, and $N(C_4)>1$.
On the other hand, the action of $C_2\cong G/C_{12}$ on $A[2](\kk)/C_{12}$ fixes the image of a point with the stabilizer $C_4$; therefore, $N(Q_8)=3$, and $N(C_4)=12$.

Finally, let $G=\ESL(\FF_3)$, and let $H=\SL_2(\FF_3)$. Since the action of $C_2\cong G/H$ on
$A[2](\kk)/H$ is trivial, $N(G)=1$, $N(Q_{16})=3$, and $N(C_4)=12$. 
\end{proof}

\begin{rem}
There is a misprint for the cyclic group of order $8$ in~\cite{Ka}.
\end{rem}

%The next corollary is a generalization of~\cite[Corollary ]{Ka}.

\begin{corollary}\label{Katsura} 
Let $A$ be an abelian surface over a perfect field $k$ of characteristic $p>2$. 
Assume that there exists a rigid and symplectic action on $A$ of a cyclic group $G$
of order $n\in\{5,8,12\}$. Then $A$ is supersingular.

If $p$ does not divide $n$, then $\FF_{p^2}\subset k$.
\end{corollary}
\begin{proof}
Let $X$ be the minimal resolution of singularities of $A/G$.
   If $p$ does not divide the order of $G$, then, according to the Katsura Theorem~\ref{Katsura_quotient}, and Proposition~\ref{Sing}, $X$ is a supersingular $K3$ surface. By~\cite[Lemma 4.4]{Ka}, $A$ is also supersingular. 

    If $p$ divides the order of $G$, we have to consider two cases:
    \begin{enumerate}
	\item $n=p=5$;
	\item $p=3$, and $n=12$.
	\end{enumerate}
    As in the proof of Lemma~\ref{L1}, $M(A)$ is a free module over the DVR $W(k)(\zeta_n)$.  If the slopes of $A$ are different, then in the decomposition of $p\ZZ[\zeta_n]$ to the product of prime ideals there are at least two different primes.
Since $\ZZ[\zeta_n]$ has only one prime ideal over $p$, the surface $A$ is supersingular.

    Let $\HH=\End^\circ(A)$. Assume that $\FF_{p^2}\not\subset k$.
    Since $\QQ[G]^\rig\cong\QQ(\zeta_n)$ is a subalgebra of $\HH$, according to Corollary~\ref{finite_field}, there exists an abelian variety $B$ over a finite field $k_1$ such that $\FF_{p^2}\not\subset k_1$ and
    \[\QQ(\zeta_n)\subset\End^\circ(A)\subset\End^\circ(B).\]
     According to Theorems~\ref{th_curv} and~\ref{th_surf}, we have the following possibilities: 
\begin{enumerate}
	\item $\HH\cong\QQ(\zeta_n)$ is a field;
	%\item $\HH\cong\Mat(2,L)$, where $L$ is a quadratic extension of $\QQ$ that is unramified in $p$;
	\item $\HH\cong \HH_\infty(\QQ(\sqrt{p}))$;
	\item $\HH\cong M(2,\QQ(\sqrt{-p}))$.
\end{enumerate}
In the first case, according to Lemma~\ref{NoSS}, $\FF_{p^2}\subset k$. A contradiction.
In cases $(2)$ and $(3)$ the central simple algebra $\HH$ of dimension $8$ over $\QQ$ contains  $\QQ(\zeta_n)$ as a subfield of degree $4$ over $\QQ$; therefore, the center of $\HH$ is a subfield  of $\QQ(\zeta_n)$, and $\sqrt{\pm p}\in\QQ(\zeta_n)$.
This is impossible if $p$ does not divide $n$.  
    \end{proof}

Now we compute the Frobenius action on the graphs of exceptional curves on $X(A,G)$.
Let $x\in A(\FF_{q^r})$ be a point of degree $r$, i.e., a morphism \[x:\Spec(\FF_{q^r})\to A.\] 
Let $S(x)$ be the stabilizer of $x$. For any $g\in S(x)$ there exists an element $\Bar g\in\Gal(\FF_{q^r}/\FF_q)$
 such that \[g\circ x=x\circ \Bar g.\]
We obtain a homomorphism \[\phi_x:S(x)\to \Gal(\FF_{q^r}/\FF_q)\] given by the formula: $\phi_x(g)=\Bar g$.

\begin{prop}\label{FActionOnSingGraph}
\begin{enumerate} 
\item  Let $A$ be an abelian surface over $\FF_q$ with a rigid and symplectic action of a cyclic group $G$ of order $n$ such that $(n,p)=1$.
 Let $x\in A(\FF_{q^r})$ be a point of degree $r$ fixed by $G$ such that $\phi_x$ is surjective.
Then the image of $x$ is an $\FF_q$-point,
the quotient singularity at $x$ is of type $A_{n/r-1}$,
and the Frobenius action on the graph of exceptional curves
 is non-trivial if and only if $\zeta_{n/r}\not\in\FF_{q^r}$.
\item Assume that $p>2$. Let $G=Q_8$, and let $A$ be an abelian surface over $\FF_q$ with a 
rigid and symplectic action of $G$.
%such that there is a rigid and symplectic action of $G$ on $A_k=A\otimes_{\FF_q}k$, where $k=\FF_q(\zeta_4)$. Assume that the action of a cyclic subgroup of order $4$ is defined over $\FF_q$, and the quotient $A_k/Q_8$ is defined over $\FF_q$.
 Then the Frobenius action on the graph $D_4$ of the quotient singularity at the origin of $A$ is trivial.
\end{enumerate}
\end{prop}
\begin{proof}
\begin{enumerate} 
\item The algebra of the quotient singularity at $x$ is given by the formula
\[S=(\oplus_s (T_x^*(A))^{\otimes s})^G.\] 
Let $k=\FF_{q^r}(\zeta_{n/r})$. We will compute $S_k=S\otimes_{\FF_q}k$.

Let $H=\Gal(\FF_{q^r}/\FF_q)$. By assumption, the homomorphism $\phi_x:G\to H$ is surjective, and we can choose a generator $g\in G$ such that $\phi_x(g)\in H$ is the Frobenius automorphism. There is an isomorphism of algebras \[\FF_{q^r}\otimes_{\FF_q}k\cong\oplus_{h\in H} k_h\] such that the Galois group $\Gal(k/\FF_q)$ acts through the natural surjection
$\pi:\Gal(k/\FF_q)\to H$: 
namely, if $\lambda\in \Gal(k/\FF_q)$, then $\lambda(k_h)=k_{\pi(\lambda) h}$.
We obtain a corresponding decomposition for 
\[T_x^*(A)\otimes_{\FF_q}k\cong \oplus_{h\in H}V_h,\]
where $V_h$ is a two-dimensional $k_h$-vector space with a semilinear action of $\ker\pi$.
The action of $G$ is given by the formula $g(V_h)=V_{\phi_x(g)h}$
 and the action of $\ker\phi_x$ on $V_h$ is $k$-linear. 
 
 Let $V=V_e$, where $e\in H$ is the trivial element. The natural injection 
\[(\oplus_sV^{\otimes s})^{\ker\phi_x}\to S_k\] is a $\ker\pi$-equivariant isomorphism. 

Since $g^r\in\ker\phi_x$,
 there exists a basis $u,v$ of $V$ such that $g^rv=\zeta v$, and $g^ru=\zeta^{-1}u$,
  where $\zeta\in k$ is a primitive $n/r$-th root of unity. 
Let $\lambda$ be a generator of $\ker\pi$. Then $\lambda(v)=\alpha u$ for some $\alpha\in k$ if and only if $\lambda(\zeta)\neq\zeta$, i.e, if and only if $\zeta\not\in\FF_{q^r}$.
In what follows, we assume that $\alpha=1$.

It is straightforward to check that \[x=v^{n/r},\; y=u^{n/r}\text{, and }z=uv\]
generate $\ker\phi_x$-invariants of $\oplus_sV^{\otimes s}$ over $k$, 
and the equation of the singularity is $xy-z^{n/r}=0$.
It follows that the Galois action on the exceptional graph becomes trivial over $k$.
Moreover, the Galois action on the exceptional graph of the blow up is non-trivial
if and only if $\lambda(x)=y$, i.e, if and only if $\zeta_{n/r}\not\in\FF_{q^r}$.

\item % Without loss of generality we may assume that the action of $j\in G$ is defined over $\FF_q$.
Let $k=\FF_q(\zeta_4)$. According to Lemma~\ref{sympl_basis3}, there exists a basis $u,v$ of $T_x^*(A)\otimes_{\FF_q}k$ such that $i$ and $j$ act as matrices
\[\left (\begin{array}{c c}
\zeta_{4} & 0\\
0 & -\zeta_{4}\\	
\end{array}\right)\quad\text{ and }\quad
\left (\begin{array}{c c}
0 & 1\\
-1 & 0\\	
\end{array}\right).\]

As before, the Galois group of $k$ over $\FF_q$ interchanges $v$ and $\alpha u$ for some $\alpha\in \FF_q^*$, and in what follows we may assume that $\alpha=1$.
The $G$-invariants are generated by $x=v^4+u^4$, $y=v^2u^2$, and $z=\zeta_4vu(v^4-u^4)$. 
Clearly, these functions are defined over $\FF_q$.
In these coordinates, the equation of the singularity is given by $z^2+y(x^2-4y^2)=0$. The exceptional divisor of the blow up of the singular point is a projective line with three singular points of type $A_1$ defined over $\FF_q$. 
Therefore, the Galois action on the resolution is trivial.
\end{enumerate}
\end{proof}

\section{A refinement of the Katsura theorem.}\label{main_sec}
In this section we use Theorem~\ref{main_cor} to obtain a classification of finite groups that act on abelian varieties over a given finite field $k$ of characteristic $p$. 

\begin{prop}\label{alg_inj}
An algebra $\HH$ from Theorem~\ref{rigid_actions} admits a homomorphism to $M(2,\HH_p)$ if and only if $p$ satisfies the conditions in the table below.

\begin{center}
\begin{longtable}{|c|c|c|c|}
\hline 
$\HH$  & $p$ & $\HH$ & $p$  \\
\hline 
$\QQ[\zeta_4]$  &  $p>0$ & $\HH_2$  &  $p>0$  \\
$\QQ[\zeta_3]=\QQ[\zeta_6]$ & $p>0$  & $\HH_3$    &  $p>0$ \\
$\QQ[\zeta_5]=\QQ[\zeta_{10}]$ &  $p\not\equiv 1\pmod 5$ & $\HH_\infty(\QQ(\sqrt{5}))$ & $p\not\equiv \pm 1\pmod 5$\\%, $p\equiv 3\pmod 5$ \\
$\QQ[\zeta_8]$ & $p\not\equiv 1\pmod 8$ &   $\HH_\infty(\QQ(\sqrt{2}))$ & $p\not\equiv \pm 1\pmod 8$\\%, $p\equiv 5\pmod 8$\\
$\QQ[\zeta_{12}]$   & $p\not\equiv 1\pmod{12}$  & $\HH_\infty(\QQ(\sqrt{3}))$ &  $p\not\equiv \pm 1\pmod{12}$\\%, $p\equiv 7\pmod 12$                    \\
\hline
%  &  $M(2,\QQ[i]))$ & \\
%  &  $M(2,\QQ[\omega]))$ & \\
%  & $M(2,\QQ(\sqrt{5}))$ &  \\
\end{longtable}
\end{center}
\end{prop}

\begin{proof}
The first column of the table follows from Theorem~\ref{trivial_H} and the fact that $\QQ(\zeta_n)\otimes_\QQ\HH_p$ represents the trivial element of the Brauer group if and only if $p$ does not split completely in $\QQ(\zeta_n)$.

Let $r\in \{2,3\}$, but $r\neq p$. Put $K=\QQ(\sqrt{-rp})$. Clearly, both $K\otimes \HH_r$, and $K\otimes \HH_p$ are trivial.
According to Theorem~\ref{trivial_H}, there is a homomorphism $K\to \HH_p$, and we have a sequence of homomorphisms of $\QQ$-algebras:
\[\HH_r\to \HH_r\otimes K\cong M(2,K)\to M(2,\HH_p).\]

Finally, let $r\in \{2,3,5\}$. We prove that there is a homomorphism of $\HH_\infty(\QQ(\sqrt{r}))$ to $M(2,\HH_p)$ if and only if $p$ does not split in $\QQ(\sqrt{r})$.
Indeed, there exists such a homomorphism if and only if there is a homomorphism in the opposite direction of the centralizers of these algebras in $M(8,\QQ)$, i.e., $\HH_p\to \HH_\infty(\QQ(\sqrt{r}))$.
The result follows from Corollary~\ref{Hp_hom}.
\end{proof}

\begin{rem}
There could be non-conjugate homomorphisms from $\QQ[G]^\rig$ to $M(2,L)$, where $L$ is quadratic over $\QQ$, but 
by the Skolem--Noether theorem, all homomorphisms from $\QQ[G]^\rig$ to $M(2,\HH_p)$ are conjugate.
\end{rem}

\begin{thm}\label{SSAVp2}
Assume that $\FF_{p^2}$ is a subfield of a perfect field $k$, and that the order of $G$ is prime to $p$ and greater than $2$. 
If there exists a supersingular abelian surface $A$ over $k$ with a rigid and symplectic action of a finite group $G$, then $G$ and $p$ satisfy the following conditions: 

\begin{enumerate}
    \item $G$ is $C_3,C_4,C_6, Q_8, Q_{12}$, or $\SL_2(\FF_3)$;
    \item if $G$ is $C_5, C_8, C_{10}, C_{12}, Q_{16}, Q_{20}, Q_{24},\CSU_2(\FF_3)$, or
    $\SL_2(\FF_5)$, that is, $G$ contains a subgroup of order $n\in\{5,8,12\}$, then 
    $p\not\equiv \pm 1\pmod n$.
\end{enumerate}

Vice versa, if $p$ and $G$ satisfy $(1)$ or $(2)$, then there exists a supersingular abelian surface over $\FF_{p^2}$ with a rigid and symplectic action of a finite group $G$.
\end{thm}

First, we prove several lemmas in a slightly more general situation. Let $k$ be a perfect field, and let $A$ be an abelian surface over $k$ with a rigid action of a finite group $G$. If $p$ does not divide $n$, we denote by $\Bar\zeta_n\in \kk$ a primitive root of unity of degree $n$ such that $\Bar\zeta_n$ lifts to $\zeta_n\in W(\kk)$.

\begin{lemma}\label{sympl_action}
Assume that $k$ is finite.
Let $g\in G$ be an element of order $p$. Then the image of $g$ in $\GL(T_0^*(A))\cong \GL_2(k)$ belongs to $\SL_2(k)$. In particular, the action of $g$ is symplectic.
\end{lemma}
\begin{proof}
The index of $\SL_2(k)$ in $\GL_2(k)$ is prime to $p$.
\end{proof}

\begin{lemma}\label{sympl_basis2}
Assume that $\Bar\zeta_n\not\in k$, and the algebraic closure of $\FF_p$ in $k$ is a finite field $\FF_q$.
Let $g\in G$ be an element of order $n\in \{3,4,5,6,8,10,12\}$ prime to $p$. 
Then the action of $g$ on $T_0^*(A)$ is symplectic if and only if $q\equiv -1\pmod n$.
\end{lemma}
\begin{proof}
The cotangent space is a one-dimensional vector space over 
$k(\Bar\zeta_n)$. The Galois group of $k(\Bar\zeta_n)$ over $k$ is of order $2$ and the action on $\Bar\zeta_n$ is given by the formula:
\[\Bar\zeta_n\mapsto\Bar\zeta_n^q.\]
Therefore, the eigenvalues of a generator of $g$ are $\Bar\zeta_n$ and $\Bar\zeta_n^q$. 
In other words, the action is symplectic if and only if $q\equiv -1\pmod n$.
\end{proof}

\begin{lemma}\label{sympl_basis}
Let $A$ be a non-ordinary abelian surface, and let $g\in G$ be an element of order
$n\in \{3,4,5,6,8,10,12\}$ prime to $p$. Assume that $\Bar\zeta=\Bar\zeta_n\in k$. 
\begin{enumerate}
    \item If $p^2\equiv 1\pmod n$, then there exists a basis $v_1,v_2, u_1, u_2$ 
    of $M(A)[\frac{1}{p}]$ such that
\[g(v_1)=\zeta v_1,\quad g(v_2)=\zeta^p v_2,\quad g(u_1)=\zeta'u_1,\quad g(u_2)=\zeta'^p u_2, \eqno{(*)}\]
 where $\zeta'$ is a prime to $n$ power of $\zeta$. 
If the action on $T_0^*(A)$ is symplectic, then there exists a basis such that $(*)$ holds with $\zeta'=\zeta^{-1}$.
\item If the action of $g$ on $T_0^*(A)$ is symplectic, and $n\in \{5,8,10,12\}$, then $p\not\equiv \pm 1\pmod n$.
\item  If $n\in \{3,4,6\}$, or $n\in \{8,12\}$, and $p\not\equiv \pm 1\pmod n$, then we 
can choose a basis of $M(A)[\frac{1}{p}]$ such that $(*)$ holds with $\zeta'=\zeta^{-1}$, and
\[F(v_1)=v_2,\text{ and } F(u_1)= u_2.\]
\end{enumerate}
\end{lemma}
\begin{proof}
There exists a basis $\bar v_1, \bar u_1$ of the cotangent space such that
\[g(\bar v_1)=\Bar\zeta \bar v_1\text{, and } g(\bar u_1)=\Bar\zeta'\bar u_1,\]
where $\Bar\zeta'\in k$ is a power of $\Bar\zeta$.
We can lift this basis to elements $v_1$ and $u_1$ of $M(A)$ such that
\[g(v_1)=\zeta v_1\text{, and } g(u_1)=\zeta'u_1,\] where $\zeta'\in W(k)$ lifts $\Bar\zeta'$.
Put $v_2=F(v_1)$, and $u_2=F(u_1)$. Then \[g(v_2)=\zeta^p v_2\text{, and }g(u_2)=(\zeta')^p u_2.\]

Assume that $p^2\equiv 1\pmod n$. Then the equality $\zeta^p=\zeta'$ is equivalent to $(\zeta')^p=\zeta$, because
  $(\zeta')^p=\zeta^{p^2}=\zeta$.
%This happens if and only if $n\in\{3,4,6,8,12\}$, or if $n\in\{5,10\}$, and $p\equiv \pm 1\pmod 5$. 
It follows that, if $\zeta^p=\zeta'$, then $v_1,u_1$ generate a \Die submodule of $M(A)$ of slope $1$. By duality, slopes of $A$ are
$0$ and $1$, i.e., $A$ is ordinary. A contradiction.
Therefore, $\zeta^p\neq\zeta'$, and $v_1, v_2, u_1,u_2$ is a basis of $M(A)[\frac{1}{p}]$.
Clearly, the action of $g$ on $T_0^*(A)$ is symplectic if and only if $\zeta'=\zeta^{-1}$. 
Part $(1)$ is proved.

We now prove $(2)$. If $n\in \{5,8,10,12\}$, and $v_1, v_2, u_1,u_2$ is a basis of
$M(A)[\frac{1}{p}]$,
 then the eigenvalues of $g$ are different by rigidity, and $p\not\equiv \pm 1\pmod n$.
Suppose that $v_1, v_2, u_1,u_2$ is not a basis of $M(A)[\frac{1}{p}]$. 
According to part $(1)$, we have $p^2\not\equiv 1\pmod n$, i.e., $n\in\{5,10\}$, and $p\not\equiv \pm 1\pmod 5$. We proved that if the action of $g$ on $T_0^*(A)$ is symplectic, then $p\not\equiv \pm 1\pmod n$.

Now, we construct a basis of $M(A)[\frac{1}{p}]$ that satisfy conditions of $(3)$;
 we call such a basis \emph{symplectic}. Our assumptions imply that $p^2\equiv 1\pmod n$. We already proved that in this case a basis $v_1, v_2, u_1,u_2$ with property $(*)$ exists.

Assume that $n\in \{3,4,6\}$. According to Lemma~\ref{eigenvalues}, either $\zeta'=\zeta^{-1}$ and the basis is symplectic, or $\zeta'=\zeta^{-p}$. In the second case, the basis $v_1,v_2, u_2,F(u_2)$ is symplectic.

Assume that $n\in \{5,8,10,12\}$, and $p\not\equiv \pm 1\pmod n$. According to Lemma~\ref{eigenvalues}, all eigenvalues of $g$ are different; 
therefore, either $\zeta'=\zeta^{-1}$ and the basis is symplectic, or $\zeta'=\zeta^{-p}$.  In the second case the basis $v_1,v_2, u_2,F(u_2)$ is symplectic. The lemma is proved.
\end{proof}

\begin{lemma}\label{sympl_basis3}
Let $G=Q_{4n}$, and let $g\in G$ be an element of order $2n$. Assume that $2n$ is prime to $p$. 
%and that and let $\zeta_{2n}\in \kk$ be an eigenvalue of $g$.
Then over $k(\Bar\zeta_{2n})$ there is a basis of $T_0^*(A)$ such that generators of $G$ act as matrices
$$\left (\begin{array}{c c}
\Bar\zeta_{2n} & 0\\
0 & \Bar\zeta_{2n}^{-1}\\	
\end{array}\right)\quad\text{ and }\quad
\left (\begin{array}{c c}
0 & 1\\
-1 & 0\\	
\end{array}\right).$$

In particular, the action of $G$ is symplectic.
%Let $G$ be $\BD_2$ or $\SL_2(\FF_3)$. Then the action on $T^*(A)$ is symplectic if and only if $q\equiv -1\pmod n$.
\end{lemma}
\begin{proof}
There exists an element $j\in G$ of order $4$ such that $jgj^{-1}=g^{-1}$.
Therefore, $\Bar\zeta_{2n}^{-1}$ is also an eigenvalue of $g$, and there exists a basis 
$v,u\in T^*(A)\otimes k(\Bar\zeta_{2n})$ such that the matrix of $g$ is diagonal.
It is straightforward to show that $j(v)=\gamma u$ for some $\gamma\in k(\Bar\zeta_{2n})$, and $j(u)=-\gamma^{-1}v$. 
Therefore, the basis $v,\gamma u$ is a desired basis.
\end{proof}

\begin{lemma}\label{sympl_basis4}
Assume that there is a linear action of $G\in\{\SL_2(\FF_3),\CSU_2(\FF_3),\SL_2(\FF_5)\}$.
If $p>3$, then the action on $T_0^*(A)$ is symplectic.
\end{lemma} 
\begin{proof}
It is well known that if $p>3$, then there exists a basis of $T_0^*(A)\otimes_k\kk$ such that $\SL_2(\FF_3)$ is generated by $Q_8$ and the matrix 

$$\frac{1}{2}\left (\begin{array}{c c}
1+\zeta_4 & 1+\zeta_4\\
-1+\zeta_4 & 1-\zeta_4\\	
\end{array}\right);$$

and $\CSU_2(\FF_3)$ is generated by $\SL_2(\FF_3)$ and the matrix 

$$\left (\begin{array}{c c}
\zeta_8 & 0\\
0 & \zeta_8^{-1}\\	
\end{array}\right).$$
These matrices are symplectic; therefore, a symplectic basis for $Q_8$ is also a symplectic basis for $\SL_2(\FF_3)$ and $\CSU_2(\FF_3)$.

The group $G=\SL_2(\FF_3)$ is simple; therefore, for any two-dimensional representation, the induced homomorphism $\det:G\to k^*$ is trivial and the representation is symplectic.   
\end{proof}

\begin{proof}[Proof of Theorem~\ref{SSAVp2}.]
%If there exists a supersingular abelian surface with a rigid and symplectic action of a group $G$, then, according to Corollary~\ref{GisSLsubgroup}, there exists a supersingular abelian surface $B$ over $\FF_{p^2}$ with a rigid action of $G$. 
%It follows that $G$ and $p$ satisfy one of the the conditions by Theorem~\ref{rigid_actions}, and Proposition~\ref{alg_inj}.
Let $A$ be an abelian surface over $k$ with a rigid and symplectic action of $G$. Then the conditions follow from Theorem~\ref{rigid_actions} and Lemmata~\ref{sympl_basis2}, and~\ref{sympl_basis}.(2).

Let us prove the existence part. 
According to Theorem~\ref{th_curv}, there exists a supersingular elliptic curve $E$ over $\FF_{p^2}$ such that $\End^\circ(E)\cong\HH_p$;
 it follows that $\End^\circ(E^2)\cong M(2,\HH_p).$  
By Theorem~\ref{main_cor}, there exists a supersingular abelian surface with an action of a group $G$ if and only if there exists a homomorphism $\QQ[G]^\rig\to M(2,\HH_p)$. 
According to Theorem~\ref{rigid_actions}, and Proposition~\ref{alg_inj}, the conditions $(1)$ and $(2)$ imply that such a homomorphism exists.

Let $A=E^2$, and let $\HH=\QQ[G]^\rig$. 
We are going to prove that under our conditions there exists a submodule $M$ of
$M(A)[\frac{1}{p}]$ such that the action on $M/FM$ is symplectic and 
\[M\left[\frac{1}{p}\right]\cong M(A)\left[\frac{1}{p}\right];\] we will call such a submodule \emph{symplectic}. 
Assume that a symplectic submodule exists. Then, according to Lemma~\ref{lem_on_Dmod}, 
there exist an abelian variety $B$ and a $p$-isogeny $\phi:A\to B$ such that $M(B)\cong M$;
in particular, for all $\ell\neq p$ we have $T_\ell(A)\cong T_\ell(B)$. Therefore, the natural homomorphism $G\to\End^\circ(B)$ factors through $\End(B)$, and $G$ acts on $B$.
By construction, the isogeny $\phi$ is $G$--equivariant, and the action of $G$ on $B$ is rigid and symplectic.

Now, we construct a symplectic submodule. Firstly, assume that $G$ is cyclic of order $n$. 
Let $\HH=\QQ(\zeta_n)$ be a quadratic field extension of $\QQ$. 
If $\Bar\zeta_n\not\in k$, then the action is always symplectic according to Lemma~\ref{sympl_basis2}.
If $\Bar\zeta_n\in k$, then, according to Lemma~\ref{sympl_basis}, we can choose a symplectic submodule $M$ of $M(A)[\frac{1}{p}]$.

Assume that $\HH=\QQ(\zeta_n)$ is a quartic field extension of $\QQ$, that is, $n\in \{5,8,10, 12\}$. 
Then either $n$ divides $p^2-1$, or $\HH\cong\QQ(\zeta_5)$. 
In the first case, Lemma~\ref{sympl_basis} implies that there exists a symplectic submodule $M$ in $M(A)$. If $\HH\cong\QQ(\zeta_5)$, then according to Lemma~\ref{sympl_basis2} the action is symplectic if and only if $p^2\equiv -1\pmod 5$, that is, $p\not\equiv \pm 1\pmod 5$.

The remaining cases follow from Lemmata~\ref{sympl_basis3} and~\ref{sympl_basis4}.
\end{proof}

\begin{prop}\label{A_H}
Let $\HH$ be a central simple algebra over $\QQ$, and let $A$ be an ordinary abelian variety over $k$ with a rigid action of a finite group $G$ such that the action factors through $\HH:$
\[\QQ[G]^\rig\to\HH\to\End^\circ(A).\]
Then the characteristic polynomial of the action of any $g\in G$ on $T_0^*(A)$ is congruent modulo $p$ to a power of the cyclotomic polynomial $\Phi_n$, where $n$ is the order of $g$.
\end{prop}
\begin{proof}
Let $V$ be the irreducible representation of $\HH$ over $\QQ$. The induced representation of $G$ on $V$ is without fixed points;
therefore, the characteristic polynomial of any element $g\in G$ of order $n$ is equal to some power of the cyclotomic polynomial $\Phi_n^{r_g}$.
%We may assume that $G=C_n$ is cyclic and is generated by an element $g$ of order $n$. 
According to Proposition~\ref{SSDmod}, $M(A)\cong M\oplus M^*$ is a direct sum of \Die modules. The action of $G$ on $M$ comes from the composition 
\[\QQ[G]^\rig\to \HH\otimes_\QQ\QQ_p\to\End^\circ(A)\otimes_\QQ\QQ_p\to\End(M[\frac{1}{p}]).\]
It follows that $M[\frac{1}{p}]$ is isomorphic to $(V\otimes_\QQ\QQ_p)^{r_M}$ as a $\HH\otimes_\QQ\QQ_p$-module, and the characteristic polynomial of the action of $g$ on $M[\frac{1}{p}]$ is equal to $\Phi_n^{r_gr_M}$. 
According to Proposition~\ref{SSDmod}, we have the isomorphism $T_0^*(A)\cong M/pM$.
The proposition is proved.
\end{proof}

\begin{ex}\label{A_V}
Let $V$ be a $\QQ$-representation of $G$ without fixed points, and let $\HH=M(d,\QQ)$, where $d=\dim V$. For any abelian variety $B$ over $k$ we get homomorphisms 
\[\QQ[G]^\rig\to M(d,\QQ)\to M(d,\QQ)\otimes_\QQ\End^\circ(B)\cong\End^\circ(B^d).\]
According to Theorem~\ref{main_cor}, there exists an abelian variety $B_V$ of dimension $d\dim B$ with a rigid action of $G$. 
\end{ex}

\begin{thm}\label{small_groups}
There exists an abelian surface over $\FF_p$ with a rigid and symplectic action of $G$ in the following cases:
\begin{enumerate}
	\item $G$ is a cyclic group of order $2,3,4$, or $6$;
\item $p\neq 2$, and $G$ is $Q_8$, or $\SL_2(\FF_3)$;
\item $p>3$, and $G\cong Q_{12}$.
\end{enumerate}
\end{thm}
\begin{proof}
Clearly, if $n=2$, then any abelian surface has a rigid and symplectic action of $C_2$.
Let $E$ be an ordinary elliptic curve over $\FF_p$, and let $G=C_n$, where $n\in\{3,4,6\}$.
The representation of $G$ on $V=\QQ[G]^\rig$ is a two-dimensional representation without fixed points. The abelian variety $E_V$ constructed in Example~\ref{A_V} is a surface, and, by Proposition~\ref{A_H}, the characteristic polynomial of the action of a generator on $T^*_0(E_V)$ is equal to $\Phi_n$, that is, the action is symplectic.
This proves part $(1)$.

If $G$ is not cyclic, we will construct an ordinary elliptic curve $E$ such that 
there exists a homomorphism from $\QQ[G]^\rig$ to $\End^\circ(E^2)\cong M(2,L)$, 
where $L=\End^\circ(E)$. Then, according to Theorem~\ref{main_cor}, in the isogeny class of $E^2$ there exists an abelian surface $A$ with a rigid action of $G$. If $G$ is $Q_8,Q_{12}$, or $\SL_2(\FF_3)$, then we use Proposition~\ref{A_H} with $\HH=\QQ[G]^\rig$ and find that the action on $A$ is symplectic by Lemmata~\ref{sympl_basis3}, and~\ref{sympl_basis4}. 

Let $G$ be $Q_8$ or $\SL_2(\FF_3)$. According to Corollary~\ref{Hp_hom}, if $L$ is not split at $2$, then there exists a homomorphism from $\QQ[G]^\rig\cong\HH_2$ to $M(2,L)$.  
We are going to apply Theorem~\ref{th_curv} and find an elliptic curve $E$ such that its endomorphism algebra $L$ is not split at $2$.
Indeed, there exists an elliptic curve $E$ with the Weil polynomial $t^2-t+p$.
If $p>2$, the discriminant of this polynomial is $1-4p\equiv 5\pmod 8$;
therefore, $L$ is inert at $2$.

If $G=Q_{12}$, then $\QQ[G]^\rig=\HH_3$. As before, we need an ordinary elliptic curve over $\FF_p$ such that its endomorphism algebra $L$ is not split at $3$. If $p\equiv 1\pmod 3$, then there exists an elliptic curve $E$ with the Weil polynomial $t^2-3t+p$; 
the discriminant of this polynomial is $9-4p\equiv 2\pmod 3$.  
If $p\equiv 2\pmod 3$, then there exists an elliptic curve $E$ with the Weil polynomial $t^2-t+p$; 
the discriminant of this polynomial is $1-4p\equiv 2\pmod 3$. In both cases $L$ is inert at $3$, if $p>3$.
\end{proof}

\begin{proof}[Proof of Theorem~\ref{main1}.]
The existence part follows from Theorems~\ref{SSAVp2} and~\ref{small_groups}. 
On the other hand, according to Theorem~\ref{Katsura_quotient}, the action of $G$ on $A$ is rigid and symplectic. Since $p$ does not divide the order of $G$, it follows from Theorem~\ref{rigid_actions} that the group $G$ belongs to the Katsura list from Theorem~\ref{Katsura_list}. 

If $G$ contains a cyclic subgroup of order $n\in\{5,8,12\}$, 
then, according to Corollary~\ref{Katsura}, the abelian surface $A$ is supersingular, and $\FF_{p^2}\subset k$.
In this situation, according to Theorem~\ref{SSAVp2}, we have $p\not\equiv\pm 1\pmod n$.
\end{proof}

\section{Traces of Frobenius}\label{SecZeta}
In this section, we use our methods to compute traces of Frobenius and zeta functions of generalized Kummer surfaces.

\subsection{Zeta functions of abelian varieties and generalized Kummer surfaces}
\label{s04}
Let $k=\fq$ be a finite field of order $q$, and let $X$ be an algebraic variety over $k$. 
Denote by $N_r$ the number of points on $\bar X$ defined over $\FF_{q^r}$. Then the zeta function of $X$ is the
formal power series
\[Z_X(t)=\exp(\sum_{r=1}^\infty
\frac{N_rt^r}{r}).\]

The zeta function of an abelian variety $A$ with the Weil polynomial $f_A=\prod_j (t-\pi_j)$ can be computed as follows:
\[Z_A(t)=\prod_{i=0}^{2\dim A}P_i(A,t)^{(-1)^{i+1}},\] where
\[P_i(A,t)=\prod_{j_1<\dots <j_i} (1-\pi_{j_1}\dots\pi_{j_i}t).\]

The zeta function of a Shioda supersingular $K3$ surface is given by the formula
\begin{equation}\label{zeta}
    Z_X(t)=\frac{1}{(1-t)P_2(qt)(1-q^2t)},
\end{equation}
where
\[
P_2(t)=\det(1-Ft|\NS(\XX)\otimes\QQ)=\prod_r\Phi_r(t)^{\lambda_r}
\]
is the characteristic polynomial of Frobenius automorphism on $\NS(\XX)$. 
Since the roots of this polynomial are roots of unity, we can write it as a product of cyclotomic polynomials $\Phi_r$.
Therefore, the zeta function is uniquely determined by $P_2$, and this polynomial is determined by the numbers $\lambda_r$. 
We denote the zeta function by 
\[(1^{\lambda_1\deg\Phi_1},\dots, r^{\lambda_r\deg\Phi_r}),\]
that is, the power of $r$ is equal to the number of roots of $P_2(t)$ that are primitive roots of unity of degree $r$.

From the definitions it follows that the number of points on a supersingular $K3$ surface $X$ over $\FF_q$ is equal to \[1+q\Tr_X +q^2,\]
 where $\Tr_X$ is the trace of the Frobenius action on $\NS(\XX)$.

\subsection{Traces of Frobenius over $\FF_{p^{2r}}$.} 

The following lemma is well known.

\begin{lemma}\label{ZetaExCurve}
Let $X=X(A,G)$ be a generalized Kummer surface. Then 
\[\NS(\XX)\otimes\QQ\cong\oplus_E\QQ E\oplus\NS(\overline{A})^G\otimes\QQ,\]
where the sum is over the exceptional lines $E$ of the resolution of singularities of $A/G$.
\end{lemma}

First, we compute the zeta function of the quotient $A/G$.

\begin{lemma}\label{zeta_lemma}
Let $A$ be an abelian surface over $\FF_q$ with the Weil polynomial $f(t)=(t^2 +\eps q)^2$ with a rigid action of a group $G$, where $\eps=\pm 1$. 
Denote the characteristic polynomial of the Frobenius action on $H^2(\Bar A,\QQ_2)^G$ by $h(t)$.
\begin{enumerate}
	\item If $G$ is cyclic of order $3,4$, or $6$, then $h(t)=(t-q)^2(t+q)^2$.
	\item If $G=Q_8$, $Q_{12}$, or $\SL_2(\FF_3)$, then $h(t)=(t+\eps q)^2(t-\eps q)$. 
       %\item If $G=C_4$, and there is an action of $Q_8$ over $\FF_{p^2}$ satisfying conditions $(1)$ and $(2)$ 
	%of the proof of Theorem~\ref{SSZeta}, then $h(t)=(t+\eps p)(t-\eps p)^2$.
\end{enumerate}
\end{lemma}
\begin{proof}
Let $V=H^1(\Bar A,\QQ_2)\otimes\QQ(\sqrt{-\eps q})$. 
Then $V=V_1\oplus V_2$, where $F$ acts as $(-1)^s\sqrt{-\eps q}$ on $V_s$, and $s\in\{1,2\}$. 
Since the action of $G$ commutes with the Frobenius action, both $V_1$ and $V_2$ are $G$-invariant, and
\[\wedge^2(V)^G=\wedge^2 V_1\oplus \wedge^2 V_2\oplus (V_1\otimes V_2)^G.\]

If $G$ is a cyclic group, then $\dim (V_1\otimes V_2)^G=2$. This proves (1). 
In case (2) it is not hard to check that $\dim (V_1\otimes V_2)^G=1$.  %It follows that $h(t)=(t+\eps q)^2(t-\eps q)$.
%In case (3), we have $j\in Q_8$ of order $4$ such that $V_1$ and $V_2$ are $j$-invariant, and $i\in Q_8$ such that $i(V_1)=V_2$.
%It is straightforward to check that $V_1\otimes V_2$ is $i$-invariant. Thus 
%\[\wedge^2(V)^G=(\wedge^2 V_1\oplus \wedge^2 V_2)^{i=1}\oplus (V_1\otimes V_2)^{j=1}.\]
%As before, this proves that $h(t)=(t+\eps p)(t-\eps p)^2$.
\end{proof}

\begin{thm}\label{SSZeta1}
Assume that $k=\FF_q$, where $q=p^{2r}$, and $p>2$.
There exists a generalized Kummer surface $X$ over $k$ that corresponds to an abelian surface $A$ with the following parameters.
%\newpage
\begin{center}
\begin{longtable}{|c|c|c|c|c|}
\hline 
$\Tr_X$& $Z_X(t)$  & $G$ & $p$ & $f_A(t)$ \\
\hline 
$22$ & $1^{22}$	& $C_2$ & $p>2$ & $(t\pm\sqrt{q})^4$ \\
%$18$ & $1^{20},2^2$	& $C_4$ & $p>2$ & $(t^2-q)^2$ \\
$14$ & $1^{18},2^4$	& $C_2$ & $p>2$ & $(t^2-q)^2$ \\
$10$ & $1^{14},2^4,4^4$	& $C_2$ & $p>2$ & $(t^2+q)(t\pm\sqrt{q})^2$ \\
%$8$ & $1^{15},2^7$	& $C_4$ & $p>2$ & $(t^2-q)^2$ \\
$6$ & $1^{14},2^8$	& $C_2$ & $p>2$ & $(t^2\pm q)^2$  \\
$4$ & $1^{10},3^{12}$ & $C_2$ & $p>2$ & $(t^2\pm\sqrt{q}t+q)^2$ \\
$2$ & $1^{12},2^{10}$ & $C_2$ & $p>2$ & $(t^2-q)^2$ \\
$0$ & $1^6,2^4,3^8,6^4$ & $C_2$ & $p\not\equiv 1\pmod{12}$ & $t^4-qt^2+q^2$ \\
\hline
\end{longtable}
\end{center}
\end{thm}
\begin{proof}
  Follows from Theorem~\ref{th_surf}, and~\cite[Theorem 7.1]{Ry}.
   \end{proof}

\subsection{Traces of Frobenius over $\FF_{p^{2r+1}}$.} 
In this section $k=\FF_{p^{2r+1}}$ is an odd extension of $\FF_p$.
We start this section with a nonexistence result over $k$.

\begin{thm}~\cite{Ar}
Suppose $p\neq 2$. If a supersingular $K3$ surface $X$ is defined over $\FF_q$ and Frobenius acts trivially on $\NS(\XX)$, then $\FF_{p^2}\subset\FF_q$.
\end{thm}

\begin{corollary}
If $X$ is a Shioda supersingular $K3$ surface over $k=\FF_{p^{2r+1}}$, then $\Tr_X\neq 22$.
\end{corollary}

\begin{corollary}
If $q$ is an odd power of $p$, then in the product
$P_2(t)=\prod_r\Phi_r(t)^{\lambda_r}$ there exists an even $r$ such that $\lambda_r>0$.
\end{corollary}

\begin{proof}
Suppose that for all even $r$ we have $\lambda_r=0$. 
There exists an odd $m$ such that $F^m$ acts trivially on $\NS(\XX)$.
We get a contradiction to the Artin Theorem for the surface $X\otimes_{\FF_q}\FF_{q^m}$.
\end{proof}

\begin{corollary}
Let $X$ is a supersingular $K3$ surface over $k$ with $\Tr_X=21$. Then the zeta function of $X$ is equal to $(1^{20},6^2)$.
\end{corollary}

\begin{question}
Is it true that $\Tr_X\neq 21$, and $\Tr_X\neq 19$ for any supersingular $K3$ surface over $\FF_q$, where $q$ is an odd power of $p$? 
\end{question}

We now prove an analog of Theorem~\ref{SSAVp2} over a finite field of odd degree.

\begin{thm}\label{SSAVp}
Let $k=\FF_{p^{2r+1}}$. Assume that the order of $G$ is greater than $2$ and is not divisible by $p$. There exists a supersingular abelian surface with the Weil polynomial $f$ and a rigid action of a group $G$ if and only if $p$, $f$, and $G$ satisfy one of the following conditions.
\begin{enumerate}
\item If $f(t)=t^4+q^2$, then $G=C_4$.
\item If $f(t)=t^4\pm qt^2+q^2$, then $G=C_3$, or $G=C_6$.
%\item If $p=5$, and $f=t^4\pm 5^{r+1}t^3+3qt^2\pm 5^{3r+2}t+q^2$, then $G=C_5$.%, 
\item If $f(t)=(t^2-q)^2$, then $G$ and $p$ satisfy the following conditions:
\begin{center}
\begin{longtable}{|c|c|}
\hline 
$G$  &  $p$  \\
\hline 
$C_3$, $C_4$, $C_6$ & $p>2$  \\    
 %$C_5$, $C_{10}$ &  $p=5$  \\
%$C_{12}$   & $p=3$  \\
$Q_8$  &  $p\not\equiv 1\pmod 8$      \\  
$Q_{12}$ & $p\equiv 2\pmod 3$   \\
%$Q_{20}$ & $p=5$   \\
$\SL_2(\FF_3)$  &  $p\not\equiv 1\pmod 8$          \\  
%$\SL_2(\FF_5)$ & $p=5$  \\
\hline
\end{longtable}
\end{center}

\item If $f(t)=(t^2+q)^2$, then $G$ and $p$ satisfy the following conditions:

\begin{center}
\begin{longtable}{|c|c|}
\hline 
$G$  &  $p$  \\
\hline 
$C_3$, $C_4$, $C_6$ & $p>2$  \\    
%$C_{12}$   & $p=3$  \\
$Q_8$  &  $p\not\equiv  -1\pmod 8$    \\  
$Q_{12}$ & $p\equiv 1\pmod 3$   \\
%$Q_{24}$ &   $p=3$   \\
$\SL_2(\FF_3)$  &  $p\not\equiv -1\pmod 8$    \\  
%$C_3\times Q_8$  &  $p=3$    \\  
\hline
\end{longtable}
\end{center}
\item $f(t)=(t^2\pm 3^{r+1}t+q)^2$, and $p=3$; in this case, 
$\QQ[G]^\rig \subset M(2,\QQ(\zeta_3))$;
\item $f(t)=(t^2\pm 2^{r+1}t+q)^2$, and $p=2$; in this case, 
$\QQ[G]^\rig \subset M(2,\QQ(\zeta_4))$.
\item $f(t)=t^4\pm 2^{r+1}t^3+qt^2\pm 2^{r+1}qt+q^2$, and $p=2$; in this case, $G\cong C_3$.
\end{enumerate}
In cases $(1)-(4)$, there exists an abelian surface with a rigid and symplectic action of the group $G$.
\end{thm}
\begin{proof}
Suppose that $G$ acts on a supersingular surface $A$ over $k$. %Let $\HH=\QQ[G]^\rig$.
The surface $A$ is either simple or isogenous to a product of two elliptic curves, say $E_1$ and $E_2$. If $E_1$ and $E_2$ are not isogenous, and the action of $G$ is rigid, there are homomorphisms \[\QQ[G]^\rig\to\End^\circ(E_i)\] for $i\in\{1,2\}$. 
According to Theorem~\ref{th_curv}, there are three possibilities for $\End^\circ(E_i)$:
\begin{itemize}
    \item $\QQ(\sqrt{-p})$;
    \item $\QQ(\zeta_3\sqrt{-3})=\QQ(\zeta_3)$, and $p=3$;
    \item $\QQ(\zeta_{8}\sqrt{2})=\QQ(\zeta_4)$, and $p=2$.
\end{itemize}
Therefore, $G$ is abelian, and according to Corollary~\ref{c240}, $G$ is cyclic.
By Lemma~\ref{eigenvalues}, $\QQ[G]^\rig\cong\QQ(\zeta_n),$ where $n$ is the order of $G$; therefore, either $G$ is trivial or $p$ divides $n$. It follows that $E_1$ and $E_2$ are isogenous, and \[\End^\circ(A)\cong M(2,\End^\circ(E_1)).\]

According to Theorem~\ref{th_surf}, if $A$ is simple, then $\End^\circ(A)$ belongs to the following list:
\begin{itemize}
\item $\HH_\infty(\QQ(\sqrt{p}))$;
\item $\QQ(\zeta_8\sqrt{p})$; 
\item $\QQ(\zeta_6\sqrt{p})$;
\item $\QQ(\zeta_{12}\sqrt{p})$;
\item $\QQ(\zeta_{5}\sqrt{5})=\QQ(\zeta_5)$, and $p=5$;
\item $\QQ(\zeta_{24}\sqrt{2})=\QQ(\zeta_{12})$, and $p=2$.
\end{itemize}

We showed that the Weil polynomial $f=f_A$ of $A$ is one of the polynomials $(1)-(7)$. 

Let $A$ be an abelian variety with the Weil polynomial $f$. 
In the first two cases $L=\End^\circ(A)$ is a field: 
\begin{enumerate}
\item $L=\QQ(\zeta_8\sqrt{p})=\QQ(\zeta_4,\sqrt{2p})$, where $p\neq 2$; 
\item $L=\QQ(\zeta_6\sqrt{p})=\QQ(\zeta_3,\sqrt{p})$, or $L=\QQ(\zeta_{12}\sqrt{p})=\QQ(\zeta_3,\sqrt{3p})$, where $p\neq 3$.
\end{enumerate}
Clearly, $G$ is cyclic of order $4$ in case $(1)$, and of order $3$ or $6$ in case $(2)$.
According to Theorem~\ref{main_cor}, in these cases there exists an abelian surface with a rigid action of $G$. By Lemmata~\ref{sympl_basis}, and~\ref{sympl_basis2}, there exists an isogeny to an abelian surface with a rigid and symplectic action of $G$.

In case $(3)$ we have \[\HH=\End^\circ(A)\cong \HH_\infty(\QQ(\sqrt{p})).\] 
If $G$ contains a cyclic subgroup of order $n\in\{5,8,12\}$, then there exists a homomorphism from $\QQ(\zeta_n)$ to $\HH$. Dimension counting shows that the image contains $\QQ(\sqrt{p})$, therefore $p$ divides $n$. A contradiction. According to Theorem~\ref{rigid_actions}, the group
$G$ is either cyclic of order $3,4,6$ or isomorphic to $Q_8$, $Q_{12}$, or $\SL_2(\FF_3)$. 

We are going to construct rigid actions for these groups.
According to Theorem~\ref{trivial_H}, there are homomorphisms from $\QQ[\zeta_3]$ and $\QQ[\zeta_4]$ to $\HH$. By Lemma~\ref{sympl_basis2}, there exists an isogeny from $A$ to an abelian surface with a rigid and symplectic action of $G$.
According to Theorem~\ref{main_cor}, there is an action of $Q_8$ or $\SL_2(\FF_3)$ on some variety in the isogeny class of $A$ if and only if there is a homomorphism of $\HH_2$ to $\HH$. 
By Corollary~\ref{Hp_hom}, there is such a homomorphism if and only if $2$ does not split in $\QQ(\sqrt{p})$. In the same way, one proves that $Q_{12}$ acts on some variety in the isogeny class of $A$ if and only if $3$ does not split in $\QQ(\sqrt{p})$. According to Lemmata~\ref{sympl_basis3}, and~\ref{sympl_basis4} these actions are symplectic.

In case $(4)$ the situation is similar:  \[\HH=\End^\circ(A)\cong M_2(\QQ(\sqrt{-p})).\]
In particular, the argument from the proof of case $(3)$ gives the same list of groups.
According to Theorem~\ref{main_cor} and Theorem~\ref{trivial_H}, there are homomorphisms from $\QQ[\zeta_3]$ and $\QQ[\zeta_4]$ to $\HH$. We use Theorem~\ref{main_cor} and Corollary~\ref{Hp_hom} again and find that there is an action of $Q_8$ or $\SL_2(\FF_3)$ on some variety in the isogeny class of $A$ if and only if $2$ does not split in $\QQ(\sqrt{-p})$, and $Q_{12}$ acts on some variety in the isogeny class of $A$ if and only if $3$ does not split in $\QQ(\sqrt{-p})$. 
\end{proof}

\begin{rem}
    If in Theorem~\ref{SSAVp} we assume that the action of $G$ on $A$ is symplectic, then, according to Theorem~\ref{rigid_actions}, we get that in case $(5)$ the group $G$ is $C_4$ or $Q_8$, and in case $(6)$ $G\cong C_3$.
\end{rem}

\begin{prop}\label{C4}
Assume that $p\equiv 1\pmod 4$. 
Then there exist supersingular generalized Kummer surfaces $K(A,C_4)$ over $\FF_q$
with zeta functions $(1^{15},2^7)$ and $(1^{20},2^2)$. 
In both cases, the Weil polynomial of $A$ is equal to $(t^2-q)^2$.
\end{prop}
\begin{proof}
Let $A'$ be an abelian surface with the Weil polynomial $(t^2-q)^2$, and let $q=p^r$.
    According to Theorem~\ref{trivial_H}, there exists a homomorphism from $\QQ[\zeta_4]$ to \[\End^\circ(A)\cong\HH_\infty(\QQ(\sqrt{p})).\]
    According to Theorem~\ref{main_cor} and Lemma~\ref{sympl_basis2}, there exists an isogeny from $A'$ to an abelian surface $A$ with a rigid and symplectic action of $C_4$. Denote by $g\in\End(A)$ the image of a generator of $C_4$. 
    
Since $F$ and $g$ commute, the vector space $V_2(A)=V_1\oplus V_2$ is a sum of $2$-dimensional $F$-invariant subspaces such that $g(V_1)=V_2$. According to~\cite[ Proposition 3.5]{Ry}, 
there exists an $F$-invariant $\ZZ_2$-submodule $T'\subset V_1$ 
such that the action of $F$ on $T'/2T'$ is non-trivial, but the action of $F^2$ is trivial.
Moreover, if $q\equiv 1\pmod 4$, then there exists an $F$-invariant $\ZZ_2$-submodule $T\subset V_1$ such that $F$ acts on $T/2T\cong(\ZZ/2\ZZ)^2$ as identity. 

By Lemma~\ref{lem_on_Tate_module}, there exists an abelian surface $B$ with a $2$-isogeny $A\to B$ and such that $T\oplus gT\cong T_2(B)$. In particular, there exists a rigid and symplectic action of $C_4$ on $B$. We have $B[2](\FF_q)\cong(\ZZ/2\ZZ)^{4}$ and, according to Proposition~\ref{Sing}, the singularities of $B/G$ are $4A_3+6A_1$; therefore, since 
$q\equiv 1\pmod 4$, according to Lemma~\ref{zeta_lemma} and Proposition~\ref{FActionOnSingGraph}, the zeta function of the generalized Kummer surface $K(B,C_4)$ is $(1^{20},2^2)$.

In the same way, there exists an abelian surface $B'$ with a rigid and symplectic action of $G$ such that $T'\oplus gT'\cong T_2(B')$. A straightforward computation shows that on $B'/G$ there are exactly two $F$-invariant singular points of type $A_3$ and two $F$-invariant singular points of type $A_1$. Moreover, since the action of $F^2$ is trivial on $B'[2]$, other singular points 
are of degree $2$ over $\FF_q$. According to Proposition~\ref{FActionOnSingGraph}, there are 
$8$ $F$-invariant exceptional lines, and $5$ pairs of lines such that Frobenius acts non-trivially.
It follows that the zeta function of $K(B',C_4)$ is $(1^{15},2^7)$.
\end{proof}

\begin{thm}\label{SSZeta2}
Assume that $k=\FF_q$, where $q=p^{2r+1}$, and $p>2$.
There exists a generalized Kummer surface $X$ over $k$ corresponding to an abelian surface $A$ with the following parameters:
%\newpage
\begin{center}
\begin{longtable}{|c|c|c|c|c|}
\hline 
$\Tr_X$& $Z_X(t)$  & $G$ & $p$ & $f_A(t)$ \\
\hline 
$20$ & $1^{21},2$	& $Q_8$ & $p\equiv 3(4)$ & $(t^2-q)^2$ \\
$18$ & $1^{20},2^2$	& $C_4$ & $p\equiv 1(4)$ & $(t^2-q)^2$ \\
$18$ & $1^{20},2^2$	& $C_2$ & $p\equiv 3(4)$ & $(t^2+q)^2$ \\
$14$ & $1^{18},2^4$	& $C_2$ & $p\equiv 1(4)$ & $(t^2-q)^2$ \\
$10$ & $1^{16},2^6$	& $C_2$ & $p\equiv 3(4)$ & $(t^2+q)^2$ \\
$8$ & $1^{15},2^7$	& $C_4$ & $p\equiv 1(4)$ & $(t^2-q)^2$ \\
$6$ & $1^{14},2^8$	& $C_2$ & $p>2$ & $(t^2+q)^2$  \\
$2$ & $1^{12},2^{10}$	& $C_2$ & $p>2$ & $(t^2-q)^2$  \\
$0$ & $1^6,2^4,3^8,6^4$ & $C_2$ & $p>2$ & $t^4-qt^2+q^2$ \\
\hline
\end{longtable}
\end{center}
\end{thm}
\begin{proof}
If $G=C_2$, we apply~\cite[Theorem 7.1]{Ry}. 
If $p\equiv 1\pmod 4$, then, according to Proposition~\ref{C4}, there exists a supersingular surface $X$ with $\Tr_X=18$, and a supersingular surface with the trace equal to $8$.

Assume that $G=Q_8$, and $f_A(t)=(t^2-q)^2$.
If  $p\equiv 3\pmod 4$, then, according to Theorem~\ref{SSAVp}, there exists an abelian surface $A$ with the Weil polynomial $f$ and with a rigid and symplectic action of $Q_8$. 
By Lemma~\ref{zeta_lemma}, the characteristic polynomial of Frobenius action on $H^2(\Bar A,\QQ_2)^{Q_8}$ is equal to $(t-q)^2(t+q)$.
We construct an abelian surface $B$ with a $2$-isogeny $B\to A$ and an action of $Q_8$ such that each orbit of the action of Frobenius on $B[2](\kk)$ 
is a subset of an orbit of $Q_8$. According to Proposition~\ref{FActionOnSingGraph} and Lemma~\ref{ZetaExCurve}, the zeta function of the generalized Kummer surface $K(B,Q_8)$ is equal to $(1^{21},2)$.  

The algebra $\End^\circ(A)$ is generated by $\HH_2\cong\QQ[Q_8]^\rig$ over its center $\QQ(\sqrt{p})$.
Let $i\in\HH_2$ be the image of an element of $Q_8$ of order $4$. Then there is a relation $(iF)^2=-q$. 
Let $T$ be a free submodule of $T_2(A)$ over the algebra $R=\ZZ_2[i,y]$, where $y=(1+iF)/2$. 
Note that $T$ is $G$-invariant.
 According to Lemma~\ref{lem_on_Tate_module}, there exists $B$ such that $T_2(B)\cong T$. In particular, $F$ is equal to $i$ on $B[2](\kk)$.
 %A surface with zeta function $(1^{20},2^2)$ can be constructed in the same way using the Weil polynomial $(t^2+q)^2$, but we will use a simpler construction.
%Let $p\equiv 3\pmod 4$. According to Theorem~\ref{RyTable}, there exists an abelian surface $A$ with Weil polynomial
% $f(t)=(t^2+q)^2$ such that the zeta function of the Kummer surface $K(A,C_2)$ is $(1^{20},2^2)$.
\end{proof}


\begin{thebibliography}{99}

%\bibitem[De78]{De} Demazure M., Lectures on $p$-divisible groups, Lecture notes in mathematics 302, Springer-Verlag, Berlin, 1972.

%\bibitem[Am]{Am} S. A. Amitsur. {\it Finite Subgroups of Division Rings.} Trans. Amer. Math. Soc. 80(1955), 361--386.

\bibitem[Ar74]{Ar} M. Artin. {\it  Supersingular $K3$ surfaces.}  Annales scientifiques de l'\'Ecole Normale Sup\'erieure, S\'erie 4, Tome 7 (1974) no. 4, pp. 543-567. 

\bibitem[BFL]{BFL} B. Banwait, F. Fit\'{e}, and D. Loughran. {\it Del Pezzo surfaces over finite fields and their Frobenius traces.}
 Math. Proc. Camb. Phil. Soc. 167 (2019) 35--60.

\bibitem[Bu]{Bu} W. Burnside. {\it On a general property of finite irreducible groups of linear substitutions.} Messenger
of Mathematics, vol. 35 (1905), pp. 51-55.

\bibitem[CCO14]{LAV} C.-L. Chai, B. Conrad, F. Oort. {\it Complex multiplication and lifting problems.}  Mathematical  Surveys  and  Monographs,  vol.  195,  American Mathematical Society, Providence, RI, 2014.

\bibitem[DO12]{DM} Ding, Y.W., Ouyang, Y. {\it A simple proof of Dieudonn\'e-Manin classification theorem.}  Acta. Math. Sin.-English Ser. 28, 1553--1558 (2012).

\bibitem[Fu88]{Fu} A. Fujiki. 
{\it Finite automorphism groups of complex tori of dimension two}. Publ. Res. Inst. Math. Sci. 24 (1988), no. 1, 1--97.
 
\bibitem[GT]{FG}  M. Suzuki. {\it Group Theory I}, Grundlehren der mathematischen Wissenschaften, 247. Springer, 1982.

%\bibitem[Ha77]{Ha} R. Hartshorne.
{\it Algebraic geometry. } Graduate Texts in Mathematics, No. 52.
Springer-Verlag, New York-Heidelberg, 1977.

\bibitem[Hw21]{Hw}  W.-T. Hwang. {\it On a classification of the automorphism groups of polarized abelian surfaces over finite fields.}
Finite Fields and Their Applications, 72 (2021).

\bibitem[Ka87]{Ka}T. Katsura. {\it Generalized Kummer surfaces and their unirationality in characteristic p}.
 J. Fac. Sci. Univ. Tokyo Sect. IA Math. 34 (1987), no. 1, 1--41

\bibitem[Ka78]{Ka78} N. M. Katz, Slope filtration of F -crystals, Journ\'ees de G\'eom\'etrie Alg\'ebrique de Rennes
(Rennes, 1978), Vol. I, Ast\'erisque, vol. 63, Soc. Math. France, Paris, 1979, pp. 113--163

%\bibitem[La94]{Lang} S. Lang. {\it Algebraic number theory.} Graduate Texts in Mathematics, Springer, 1994.

\bibitem[LT20]{LT} D. Loughran, A. Trepalin. {\it
Inverse Galois problem for del Pezzo surfaces over finite fields.}
Math. Res. Lett., 27 (3) (2020), pp. 845-853

\bibitem[Ma63]{Manin63} Yu. I. Manin, {\it The theory of commutative formal groups over fields of finite characteristic}.  Uspekhi Mat. Nauk, 18:6(114) (1963), 3--90; Russian Math. Surveys, 18:6 (1963), 1--83.

\bibitem[Mil08]{Milne} J. Milne. {\it Abelian varieties.} 2008.
 http://www.jmilne.org/math/CourseNotes/av.html

\bibitem[MN02]{MN} D. Maisner, E. Nart.
{\it Abelian surfaces over finite fields as Jacobians.}
 With an appendix by Everett W. Howe. Experiment. Math.  11  (2002),  no. 3, 321--337.

\bibitem[MO]{MO}  I. Reiner. {\it Maximal  orders}. London Mathematical Society Monographs, No.5. AcademicPress, London-New York, 1975, 395 pp.

\bibitem[MP15]{MP} K. Madapusi Pera. {\it The Tate conjecture for K3 surfaces in odd characteristic.} Invent. math. 201, 625–668 (2015).

\bibitem[Mum70]{Mum} D. Mumford. {\it Abelian varieties.}
Tata Institute of Fundamental Research Studies in Mathematics, No. 5.
Oxford University Press, London 1970.

\bibitem[Oort]{Oort} F. Oort, {\it Subvarieties of Moduli Spaces.}
Inventiones mathematicae 24 (1974), 95--120.

\bibitem[Pink]{Pink} R. Pink. {\it Finite Group Schemes.} Lecture notes.  https://people.math.ethz.ch/~pink/FiniteGroupSchemes.html

\bibitem[Ri76]{Ri76} K. A. Ribet, Galois action on division points of abelian varieties with real multiplications. Am. J. Math. 98, 751–804 (1976).

\bibitem[RT16]{RT16} S. Rybakov, A. Trepalin. {\it Minimal cubic surfaces over finite fields}. 
Mat. Sb. {\bf 208} (2017), no. 9, 148--170.

\bibitem[Ru90]{Ru} H.-G, R\"uck.
{\it Abelian surfaces and Jacobian varieties over finite fields.}
Comp.\ Math. {\bf 76}, 1990, 351--366. 

\bibitem[Ry10]{Ry1} S. Rybakov.
{\it The groups of points on abelian varieties over finite fields.}
 Cent. Eur. J. Math. 8(2), 2010, 282--288. arXiv:0903.0106

\bibitem[Ry12]{Ry12} S. Rybakov. {\it The groups of points on abelian surfaces over finite fields. }
In Arithmetic, Geometry, Cryptography and Coding Theory, Cont. Math., vol. 574, Amer. Math. Soc., Providence, RI, 2012, pp. 151--158. 
arXiv:1007.0115

\bibitem[Ry14]{Ry} S. Rybakov. {\it The finite group subschemes of abelian varieties over finite fields.} Finite Fields and Their Applications. 29 (2014), 132--150. arXiv:1006.5959

%\bibitem[Ry18]{Ry18} S. Rybakov. {\it Families of algebraic varieties and towers of algebraic curves over finite fields}. Math. Notes, 104:5 (2018), 712--719. https://arxiv.org/abs/1710.05395

%\bibitem[RyGa20]{RyGa}  Galkin S., Rybakov S., {\it A family of algebraic surfaces and towers of algebraic curves over finite fields}. 
%Math. Notes., 106:6, (2019), 1014--1018.  https://arxiv.org/abs/1910.14379


\bibitem[Sch12]{Sch} Sch\" utt M., {\it A note on the supersingular $K3$ surface of Artin invariant 1.}Journal of Pure and Applied Algebra, 216 (2012), Issue 6, 1438--1441.

\bibitem[Serre12]{Ser12} J.-P. Serre. {\it Lectures on $N_X(p)$}.
Chapman \& Hall/CRC Research Notes in Mathematics, 11. CRC Press, Boca Raton, FL, 2012.

\bibitem[Ta66]{Ta66} J. Tate. {\it Endomorphisms of abelian varieties over finite fields.}
 Inventiones mathematicae 2 (1966), Issue 2, pp 134--144.

\bibitem[T20]{T20} A. Trepalin. {\it Del Pezzo surfaces over finite fields}. 
Finite Fields and Their Applications, 68 (2020).

\bibitem[Wa69]{Wa} Waterhouse W.,
{\it Abelian varieties over finite fields.}
 Ann.\ scient.\ \'Ec.\ Norm.\ Sup., 1969, 4 serie 2, 521--560.

\bibitem[We]{WE} GAP package Wedderburn Decomposition of Group Algebras (WEDDERGA) https://www.gap-system.org/Packages/wedderga.html

\bibitem[Wo11]{Wolf} Wolf J. A. {\it Spaces of Constant Curvature.(Sixth edition.).} AMS Chelsea Pub., 2011.

%\bibitem[WM69]{WM} Waterhouse W., Milne J.,
%Abelian varieties over finite fields, Proc. Sympos. Pure Math.,
%Vol. XX, State Univ. New York, Stony Brook, N.Y., 1969, 53--64.

%\bibitem[Xing]{Xi} Ch. Xing.
%{\it The structure of the rational point groups of simple abelian
%%varieties of dimension two over finite fields.} 
%Arch.\ Math. {\bf 63}, 1994, 427--430.

\end{thebibliography}
\end{document}